\documentclass[10pt,a4paper,oneside,english]{article}   
\usepackage[T1]{fontenc}
\usepackage[utf8]{inputenc}
\usepackage[english]{babel}
\usepackage{latexsym}
\usepackage{midpage}
\usepackage{newlfont}
\usepackage{color}
\usepackage{amsmath} 
\usepackage{amsfonts} 
\usepackage{amsthm} 
\usepackage{cleveref}
\usepackage{graphicx} 
\usepackage{caption}
\usepackage{subcaption} 
\usepackage{enumitem}
\usepackage{mathdots}
\usepackage{marginnote}
\theoremstyle{plain} 
\newtheorem{theorem}{Theorem}[section] 
\newtheorem{corollary}[theorem]{Corollary} 
\newtheorem{lemma}[theorem]{Lemma} 
\newtheorem{proposition}[theorem]{Proposition} 
\usepackage{xypic}
\theoremstyle{definition} 
\newtheorem{definition}{Definition}[section] 
\theoremstyle{remark} 
\newtheorem{remark}{Remark}[section]
\newtheorem{example}{Example}[section] 
\numberwithin{equation}{section}
\allowdisplaybreaks
\usepackage{longtable}
\usepackage{setspace}
\usepackage{booktabs}
\usepackage[margin=1in,left=1.04in, right=1.04in,includefoot]{geometry}
\usepackage{rotating}
\usepackage{listings}
\newcommand{\E}{\mathbb{E}} 
\newcommand{\N}{\mathbb{N}} 
\newcommand{\prob}{\mathbb{P}} 
\newcommand{\I}{\mathbb{I}} 
\newcommand{\R}{\mathbb{R}} 
\newcommand{\F}{\mathcal{F}} 
\newcommand{\G}{\mathcal{G}} 

\author{Fred Espen Benth\thanks{Department of Mathematics, University of Oslo, 0316 Blindern, Norway; fredb@math.uio.no.} \and Silvia Lavagnini\thanks{Department of Mathematics, University of Oslo, 0316 Blindern, Norway; silval@math.uio.no.}}
\title{Correlators of Polynomial Processes}
\newcommand{\bigzero}{\mbox{\normalfont\Large\bfseries 0}}
\usepackage{nicematrix}
\usepackage{arydshln}
\usepackage{tikz}

\begin{document}
\maketitle

\begin{abstract}
In the setting of polynomial jump-diffusion dynamics, we provide an explicit formula for computing correlators, namely, cross-moments of the process at different time points along its path. The formula appears as a linear combination of exponentials of the generator matrix, extending the well-known moment formula for polynomial processes. The developed framework can, for example, be applied in financial pricing, such as for path-dependent options and in a stochastic volatility models context. In applications to options, having closed and compact formulations is attractive for sensitivity analysis and risk management, since Greeks can be derived explicitly. 
\end{abstract}
 

\section{Introduction}
\label{intropol}
A jump-diffusion process is called \emph{polynomial} if its extended generator maps any polynomial function to a polynomial function of equal or lower degree. As a consequence, expectations of any polynomial in the future state of the process, conditioned on the information up to the current state, are given by a polynomial of the current state. Conditional moments can thus be calculated in closed form without any knowledge of the probability distribution nor of the characteristic function, up to the computation of the exponential of the generator matrix. The class of polynomial processes includes exponential Lévy processes and affine processes, with the Ornstein--Uhlenbeck processes as a canonical example. Moreover, polynomial jump-diffusions have been studied both in a Markovian \cite{cuchiero2011, cuchiero2012} and non-Markovian \cite{filipovic2020} contexts. We refer to \cite{filipovic2016} for a mathematical analysis on polynomial diffusions.

Because of their closed moment formula, polynomial processes have many applications in finance and one of the first is addressed in \cite{zhou2003}. In the literature, we find examples on interest rates \cite{delbaen2002, filipovic2017},  stochastic volatility models \cite{ackerer2020, ackerer2018, filipovic2016}, option pricing \cite{ackerer2020b, filipovic2020} and energy modelling \cite{kleisinger2020, ware2019}. In \cite{cuchiero2012} the properties of jump-diffusion processes are exploited to improve the performance of computational and statistical methods, such as the generalized method of moments, and for variance reduction techniques in Monte Carlo methods. Further examples cover stochastic portfolio theory \cite{cuchiero2019}.

We consider a stochastic basis $(\Omega, \F, \prob)$ with a filtration $\{\F_t\}_{t\ge 0}$ and a polynomial jump-diffusion real-valued process $Y$. For any polynomial function $p$ of degree $n$ with vector of coefficients $\vec{p}\in \R^{n+1}$ with respect to a vector basis of polynomials $H_n(x)\in \R^{n+1}$, the moment formula gives 
\begin{equation*}
	\E\left[p(Y(T))\left.\right|\F_t\right] =  \vec{p}^{\top}e^{G_n(T-t)}H_n(Y(t)), \quad 0\le t\le T,
\end{equation*}
with  $G_n\in \R^{(n+1)\times (n+1)}$ the corresponding generator matrix. In this article, we extend the framework to $m+1$ polynomial functions and study conditional expectations of the form
\begin{equation}
	\label{corrn}
	\E\left[p_m\left(Y(s_0)\right)p_{m-1}\left(Y(s_1)\right)\cdot \dots \cdot p_0\left(Y(s_m)\right)\left.\right|\F_t\right]
\end{equation}
which we call \emph{$(m+1)$-point correlators}.
Here $t < s_0 < s_1 < \dots < s_m<T <\infty$ and $p_k$ are polynomial functions of degree $n_k$, $k=0,\dots, m$. We denote by $n:=\max \left\{n_0,\dots, n_m\right\}$ the maximal degree. 

For $m=0$ equation \eqref{corrn} corresponds to computing moments of $Y$, which are given by the moment formula. Hence the $(m+1)$-point correlators can in principle be obtained for any $m>0$ by iterating the moment formula. For example, for $m=1$ one applies the tower rule for $\F_t\subseteq \F_{s_0}$ to get
\begin{equation}
	\label{corrn3p}
	\E\left[p_1\left(Y(s_0)\right) p_0\left(Y(s_1)\right)\left.\right|\F_t\right] = \E\left[p_1\left(Y(s_0)\right) q_0\left(Y(s_0); s_1-s_0\right)\left.\right|\F_t\right]
\end{equation}
where $$q_0\left(Y(s_0); s_1-s_0\right) := \E\left[p_0\left(Y(s_1)\right)\left.\right|\F_{s_0}\right]= \vec{p}_{0}^{\top}e^{G_{n_0}(s_1-s_0)}H_{n_0}(Y(s_0))$$ is the polynomial obtained by applying the moment formula to $p_0(x)$. In particular, $q_0(x;s)$ has time dependent coefficients $q_{0,k}^s$, $s\ge0$, $k=0,\dots,n_0$. The product $\tilde{p}_1(x; s):=p_1\left(x\right) q_0\left(x; s\right)$ is then a polynomial function of degree $n_0+n_1$ with time dependent coefficients given by
\begin{equation*}
	\tilde{p}_{1,j}^s= \sum_{k+i=j} p_{1,i}\,q_{0,k}^s \quad \mbox{for } j=0,\dots, n_0+n_1  \mbox{ and } s\ge0.
\end{equation*}
Another application of the moment formula, this time to $\tilde{p}_1(x; s)$, produces an expression for \eqref{corrn3p} of the form
\begin{equation*}
	\E\left[p_1\left(Y(s_0)\right) p_0\left(Y(s_1)\right)\left.\right|\F_t\right] =\vec{\tilde{p}}_{1}^{s_1-s_0\,\top} e^{G_{n_0+n_1}(s_0-t)} H_{n_0+n_1}(Y(t)).
\end{equation*}
This procedure can then be iterated for larger values of $m$. However, performing the calculations is non-trivial because of the algebraic complexity of manipulating the expressions involved. With this article we make headway on this issue by providing a fully explicit closed formula for correlators.

The key for proving the moment formula lies in the existence of the generator matrix $G_n$: for a fixed $n$ and a fixed basis vector of polynomials $H_n(x)$, this is the linear representation of the action of the extended generator on $H_n(x)$. However, for $m=1$ we must deal with the product of two basis vectors, which is an object of the form $H_n(x)H_n(x)^{\top}\in \R^{(n+1)\times (n+1)}$ and for which a generator matrix cannot be constructed. We then consider the vectorization of $H_n(x)H_n(x)^{\top}$, namely we stack the columns of $H_n(x)H_n(x)^{\top}$ into a single column vector. The matrix $H_n(x)H_n(x)^{\top}$ contains however redundant terms and so does its vectorization. For $H_n(x):=(1,x,x^2,\dots,x^n)^{\top}$, which is the case we consider here, redundant terms means repeated powers of $x$. This implies that the corresponding generator matrix contains equal rows and/or zero columns, making it impossible to generalize the framework to $m>1$.

We resolve this issue by introducing two linear operators, the first of which we call the \emph{L-eliminating matrix}. This eliminates from the vectorization of $H_n(x)H_n(x)^{\top}$ the redundant powers of $x$ and returns a vector that coincides with $H_{2n}(x)$, for which there exists the generator matrix $G_{2n}$. Using the inverse operator, called the \emph{L-duplicating matrix}, we then recover the full-dimensional vector, and, finally, via inverse-vectorization we obtain the linear operator required, which allows to compute the correlator formula for $m=1$. We summarize these steps in the following graph:
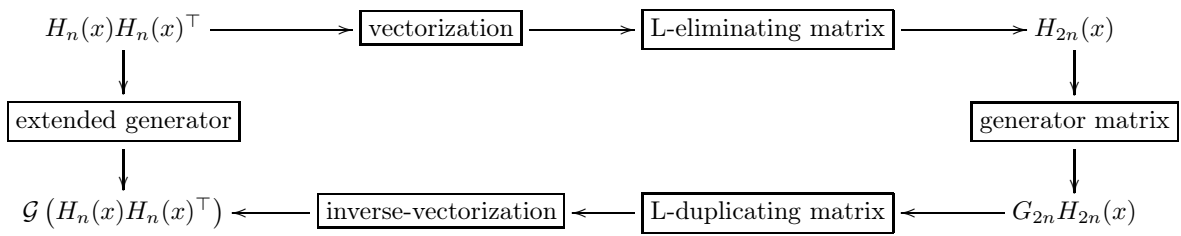
\begin{equation*}
	\resizebox{1\textwidth}{!}{
		\xymatrix@R-=0.5cm{
			H_{n}(x)H_{n}(x)^{\top} \ar[r] \ar[d] &\fcolorbox{black}{white}{vectorization} \ar[r] & \fcolorbox{black}{white}{L-eliminating matrix} \ar[r] &  H_{2n}(x)\ar[d]\\
			\fcolorbox{black}{white}{extended generator}  \ar[d]& & & \fcolorbox{black}{white}{generator matrix}\ar[d]\\
			\G \left(H_{n}(x)H_{n}(x)^{\top}\right) & \ar[l]  \fcolorbox{black}{white}{inverse-vectorization} &  \ar[l] \fcolorbox{black}{white}{L-duplicating matrix} & \ar[l] G_{2n} H_{2n}(x)}
	}
\end{equation*}
These steps work also when increasing further the number of polynomials. For $m+1>1$, we must deal with $m+1>1$ basis vectors $H_n(x)$. 
This leads to an object whose structure is more complex and requires the appropriate eliminating and duplicating matrices, for which we prove a recursion formula in the number of polynomials $m\ge 1$. With these, we compute the general correlator formula. 

As we shall see, for $H_n(x)=(1,x,x^2,\dots,x^n)^{\top}$, the matrix $H_n(x)H_n(x)^{\top}\in \R^{(n+1)\times(n+1)}$ is a so-called \emph{Hankel matrix}, for which the elements on the same skew-diagonals coincide. Hankel matrices constitute an important family of matrices that play a fundamental role in diverse fields, from computer science to engineering, mathematics and statistics \cite{peller2012}. They are indeed applied in theory of moments \cite{fasino1995, munkhammer2017}, time series analysis \cite{golyandina2001, hassani2010}, signal analysis \cite{jain2014, jain2015}, and in theory of orthogonal polynomials \cite{townsend2018} among other areas. This means that (part of) our analysis might have applications in many different fields, going beyond the polynomial jump-diffusion theory studied here. We also mention that a Hankel matrix is a "row-reversed" Toeplitz matrix, so that some of the results proved in the current article can be adapted to this other class of matrices for possibly further applications.

We point out that our correlator formula is not really an alternative to applying iteratively the moment formula, as, indeed, it strongly relies on it combined with the tower rule for $\F_t\subseteq \F_{s_0}\subseteq \cdots\subseteq\F_{s_m}$. It however provides a solution to the algebraic burden that arises when applying the moment formula directly. The correlator formula is indeed fully explicit, while getting an explicit expression is not straightforward when iterating the moment formula directly. Since having closed formulas is an advantage for example in those applications that require to differentiate, such as for computing Greeks, our approach is thus more convenient. Not surprisingly, numerical experiments show that the correlator values obtained with our formula coincide with the values obtained by iterating the moment formula. Moreover, the time costs for the two approaches is comparable up to around $m=10$ polynomials. We compare the results with a Monte Carlo approach, showing that this latter one is outperformed from a time cost point of view, in addition to exhibiting low degrees of accuracy. We stress that the correlator formula only involves linear combinations of the matrix exponential of the generator matrix. Assuming these exponential matrices to be exact, we thus have a formula for correlators which in practice is exact.

We finally provide two recursion formulas for the generator matrix and its matrix exponential.  Despite several approaches have been studied for calculating efficiently the matrix exponential of a block triangular matrix \cite{kressner2017, higham2005}, up to our knowledge, no rigorous study in terms of the building blocks has been developed yet concerning the generator matrix and its exponential. These results can then be applied for analytical purposes as mentioned before. We point out that our framework is based on the monomial basis, which appears convenient for obtaining formulas more easily and explicitly. However, it can be extended to any other polynomial basis, provided the matrix for the change of basis. For practical applications, orthogonal basis are indeed more convenient, but analytically more challenging.

The rest of the paper is organized as follows. In Section \ref{mot} we clarify the name correlators and give some financial motivations for studying them. In Section \ref{polprosssec} we introduce rigorously polynomial processes and the generator matrix. In Section \ref{correlators} we solve the two-point correlators problem, presenting the main tools and framework which allows to solve the $(m+1)$-point correlators problem in Section \ref{hcorr}. In Section \ref{recursionG} we provide two recursions for the generator matrix and its matrix exponential, together with the formula for the change of basis. Finally, in Section \ref{experiments} we consider some applications and numerical aspects and Section \ref{conc} concludes with some remarks. Appendix \ref{appApol} contains some combinatorial properties of the operators introduced in the paper and Appendix \ref{el_dup} the proofs of the main results.

\subsection{Motivations}
\label{mot}
In \cite[Section 9.3]{barndorff2018} the authors define the concept of \textit{correlator}, a standard tool in turbulence theory. For $t< s_0 < s_1<T <\infty$ and $k_0,k_1 \in \N$, the correlator of order $(k_0,k_1)$ between $Y(s_0)$ and $Y(s_1)$ is a generalization of the autocorrelation defined by
\begin{equation*}
	\mathrm{Corr}_{k_0,k_1}(s_0, s_1; t) = \frac{\E\left[\left.Y(s_0)^{k_0}Y(s_1)^{k_1}\right|\F_t \right]}{\E\left[\left.Y(s_0)^{k_0}\right|\F_t \right]\E\left[\left.Y(s_1)^{k_1}\right|\F_t \right]}.
\end{equation*}
In this article we extend this definition of correlator to any expectation like the one in equation \eqref{corrn}.

We introduce now two possible applications:  Asian option pricing and pricing in the context of stochastic volatility models. We intend to motivate our analysis, leaving details aside for future work.

\subsubsection{Path-dependent options}
\label{pathdependent}
We consider path-dependent options, such as Asian options, for which the entire path of the price process within the settlement period $[t,T]$, is taken into account by the payoff function \cite{kemna1990}. If $Y$ is the risk-neutral price dynamics of the underlying asset, $r>0$ the risk-free interest rate and $\varphi$ the payoff function, the discounted price at time $t$ for an Asian-style option settled against the discrete arithmetic average of the spot price $Y$ in the settlement period is given by
\begin{equation}
	\label{pay}
	\Pi(t) = e^{-r(T-t)}\E\left[\left.\varphi\left(\frac{1}{m+1}\sum_{j=0}^{m}Y(s_j)\right)\right|\F_t\right] \qquad \mbox{for } t<s_0 <s_1<\dots < s_m = T \mbox{ and } m\ge0.
\end{equation}
This kind of options was traded a decade ago at Nord Pool, the Nordic commodity market for electricity \cite{weron2007}. Other classes of derivatives of similar kind are calendar spread options and options on baskets of assets evaluated at different times, as well as Asian options with continuous averaging.

For $\varphi$ a real-valued continuous function on a bounded interval, we consider $\hat{\varphi}$ as the polynomial approximation of $\varphi$, e.g., by Hermite polynomials or Taylor expansions, depending on the nature of $\varphi$ itself. Then the price for the Asian option in equation \eqref{pay} is found by
\begin{equation*}
	\label{priceasian}
	\Pi(t)\approx  e^{-r(T-t)}\, \E\left[\left.\hat{\varphi}\left(\frac{1}{m+1}\sum_{j=0}^{m}Y(s_j)\right)\right|\F_t\right] = e^{-r(T-t)}\, \sum_{\boldsymbol{k}}\alpha_{\boldsymbol{k}}\E\left[\left.Y(s_0)^{k_0} Y(s_1)^{k_1} \cdots Y(s_m)^{k_m}\right|\F_t\right]
\end{equation*}
for certain coefficients $\{\alpha_{\boldsymbol{k}}\}_{\boldsymbol{k}}$ and the multi-index $\boldsymbol{k} = (k_1, \cdots, k_m)$. 
This leads to study conditional expectations of the form  $\E\left[\left.Y(s_0)^{k_0} Y(s_1)^{k_1} \cdots Y(s_m)^{k_m}\right|\F_t\right]$, which is a particular instance of equation \eqref{corrn} obtained with $p_j(x) =x^{k_j}$, $j=0,\dots, m$. In particular, in \cite{mio2021} the author derives explicit price formulas for call-style discrete average arithmetic Asian options by following the approach just described, namely by approximating the payoff function with orthogonal polynomials and by the correlator formula developed in this article.

\subsubsection{Stochastic volatility models}
\label{stochvol}
For $0\le t \le T$ we consider the process $X$ defined by $X(T) = \int_{t}^{T} \sigma(s)dB(s),$
with $B$ a standard Brownian motion and $\sigma$ a volatility process which we assume to be independent from $B$. If $\varphi$ is the payoff function and $r>0$ the risk-free interest rate, we want to price a financial derivative like follows:
\begin{equation*}
	\label{pricePi}
	\Pi(t)  = e^{-r(T-t)}\,\E \left[ \left.\varphi \left( X(T)\right)\right| \F_t \right].
\end{equation*} 
A possible approach suggested in \cite{carr1999} is to consider the Fourier transform $\hat{\varphi}$ of $\varphi$. Under appropriate integrability conditions on $\varphi$, we then write that $\varphi(x) = \int_{-\infty}^{\infty} \hat{\varphi}(z)e^{2\pi i x z} dz$, and the option price becomes
\begin{equation*}
	\label{price2Pi}
	\Pi(t)  =  e^{-r(T-t)}\, \E \left[ \left.\int_{-\infty}^{\infty} \hat{\varphi}(z)e^{2\pi i X(T) z} dz\right| \F_t \right].
\end{equation*} 
For $\sigma$ and $B$ independent, by the tower rule, we now condition with respect to the filtration $\{\F_t^{\sigma}\}_{t\ge 0}$ generated by $\sigma$ up to time $T$. The process $X(T)$ has then a Gaussian distribution with mean $0$ and variance $\int_t^T\sigma^2(s) ds$, hence
\begin{equation*}
	\label{price3}
	\Pi(t)
	 =  e^{-r(T-t)}\,\E \left[\left. \int_{-\infty}^{\infty} \hat{\varphi}(z) e^{-2\pi^2 z^2 \int_t^T\sigma^2(s) ds}  dz  \right| \F_t \right]=e^{-r(T-t)}\int_{-\infty}^{\infty} \hat{\varphi}(z) \,\E \left[\left. e^{\lambda \int_t^T\sigma^2(s) ds}  \right| \F_t \right]dz
\end{equation*}
for $\lambda\leq 0$. By considering the Taylor expansion for the exponential function, the expectation becomes
\begin{equation}
	\E \left[\left. e^{\lambda \int_t^T\sigma^2(s) ds}  \right| \F_t \right] = \E \left[\left. \sum_{k=0}^{\infty}\frac{1}{k!}\left(\lambda \int_t^T\sigma^2(s) ds\right)^k \right| \F_t \right]  = \sum_{k=0}^{\infty}\frac{\lambda^k}{k!} \E \left[\left.\left(\int_t^T\sigma^2(s) ds\right)^k \right| \F_t \right],\label{11}
\end{equation}  
that is, we need to find the moments of the integrated volatility, $\int_t^T\sigma^2(s) ds$. For $Y(s):= \sigma^2(s)$ modelled by a polynomial process, we notice that the bivariate process $\left(Y(T),\int_t^TY(s)ds\right)$ is also polynomial. Hence the moments of $\int_t^T\sigma^2(s)ds$ can be computed with the moment formula applied to this bivariate polynomial process. As an alternative approach, using iteratively the fundamental theorem of calculus, it can be proved that for every $k\ge1$ the $k$-th power of  $\int_t^T\sigma^2(s) ds$ can be rewritten in terms of a $k$-th order integral, namely
\begin{equation}
	\label{12}
	\left(\int_t^T\sigma^2(s) ds\right)^k
	= \frac{1}{k!} \int_{t}^{T}\int_t^{s_k}\cdots \int_t^{s_2} Y(s_1)Y(s_2)\cdots Y(s_k)ds_1\cdots ds_k,
\end{equation}
where $t< s_1 < s_2 < \dots < s_k< T$ is a partition of $[t,T]$. Combining equations \eqref{11} and \eqref{12}, we get
\begin{equation}
	\label{intcorr2}
	\E \left[\left. e^{\lambda \int_t^TY(s) ds}  \right| \F_t \right] = \sum_{k=0}^{\infty}\frac{\lambda^k}{(k!)^2} \int_{t}^{T}\int_t^{s_k}\cdots \int_t^{s_2} \E \left[\left.Y(s_1)Y(s_2)\cdots Y(s_k)\right| \F_t \right]ds_1\cdots ds_k,
\end{equation}
so that for every $k\ge 1$ we need to study expectations of the form of $\E \left[ \left.Y(s_1)Y(s_2)\cdots Y(s_k) \right| \F_t \right]$.

Interestingly, $Y(T)=\int_0^T\sigma^2(s)ds$ appears also in the pricing of VIX-derivatives, that is, derivatives on the realized variance and volatility. For derivatives paying $\psi(Y(T))$, we can use the Fourier approach above as long as $\psi$ is integrable, with an integrable Fourier transform $\widehat{\psi}$, to end up again with a conditional expectation as in equation \eqref{intcorr2}. The volatility swap price, i.e., the swap price on the realized volatility, is defined as the conditional expected value $\mathbb E[Y^{1/2}(T)\,\vert\,\mathcal F_t]$ for $t\leq T$. Expanding $x\mapsto\sqrt{x}\mathrm{1}_{x\geq 0}$ in the Hermite functions, being a basis for the space $L^2(\mathbb R,\gamma(x)dx)$ with $\gamma$ the standard normal density function, we obtain a series representation of the swap price in terms of conditional moments of $Y(T)$. For more details on VIX-derivatives with numerical examples based on Fourier methods, we refer to \cite{Benth-vix-paper}. Here the Barndorff-Nielsen \& Shephard stochastic volatility model is consider for $\sigma^2(t)$, which is a polynomial jump process as will be defined in the next section.

\section{Polynomial processes}
\label{polprosssec}
Following \cite{filipovic2020}, we consider a jump-diffusion operator on $\R$ of the form
\begin{equation}
	\label{genD}
	\G f(x) = b(x)f'(x) + \frac{1}{2}\sigma^2(x)f''(x)+\int_{\R}\left( f(x+z)-f(x)-f'(x)z \right)\ell(x,dz),
\end{equation} 
for some measurable maps $b:\R \to \R$ and $\sigma:\R \to \R$, and a transition kernel $\ell:\R\times \R \to \R$ such that $\ell(x, \{0\})= 0$ and $\int_{\R}|z|\land |z|^2\ell(x, dz)<\infty$ for all $x\in \R$. We then let $Y$ be the jump-diffusion stochastic process having $\G$ as extended generator. This means that for every bounded function $f:\R \to \R$ with continuous second derivative and $y\in \R$, the process $f(Y(t))-f(y) -\int_0^t \G f(Y(s))ds$ is a $\left(\F_t, \prob_y\right)$-local martingale.

We now denote with $\mathrm{Pol}(\R)$ the algebra of polynomials on $\R$ and with $\mathrm{Pol}_n(\R)$ the subspace of all polynomials of degree less than or equal to $n$ on $\R$. We say that $\G$ is \emph{well defined} on $\mathrm{Pol}(\R)$ if $\int_{\R}|z|^n\ell(x, dz)<\infty$ for all $x\in \R$ and $n\ge 2$, and $\G f(x) = 0$ for $f(x)\equiv 0$ on $\R$. We  then give the following definition of a polynomial jump-diffusion process.
\begin{definition}[Polynomial jump-diffusion process]
	We call the operator $\G$ \emph{polynomial} if it is well defined on $\mathrm{Pol}(\R)$ and it maps $\mathrm{Pol}_n(\R)$ to itself for each $n\in \N$. In this case, we call $Y$ a \emph{polynomial jump-diffusion process}.  
\end{definition}
Assuming $\G$ to be polynomial, from  \cite[Lemma D.4]{filipovic2020}, the process
\begin{equation}
	\label{martingale}
	p(Y(t))-p(y) -\int_0^t \G p(Y(s))ds \qquad \mbox{is a }\left(\F_t, \prob_y\right)\mbox{-martingale}
\end{equation}
for all $p \in  \mathrm{Pol}_n(\R)$, $y \in \R$ and $t\ge 0$. This basically means that all increments of \eqref{martingale} have vanishing expectation. Moreover, from \cite[Lemma 1]{filipovic2020}, the polynomial property of $\G$ can be characterized in terms of its coefficients: it must hold that
\begin{equation}
	\label{poldefcond_FL}
	b\in \mathrm{Pol}_1(\R),\qquad
	\sigma^2 +\int_{\R}z^2\ell(\cdot,dz)\in \mathrm{Pol}_2(\R) \quad \mbox{and}\quad
	\int_{\R}z^m\ell(\cdot,dz)\in \mathrm{Pol}_m(\R)  \mbox{ for all } m\ge 3.
\end{equation}
To fulfil these conditions, we shall assume that for every $m\ge 2$ there exist $b_0, b_1, \sigma_0, \sigma_1, \sigma_2, \xi_0^m, \dots, \xi_m^m$ real constants such that
\begin{equation}
\label{poldefcond}
b(x)=b_0+b_1x,\qquad \sigma^2(x) = \sigma_0 +\sigma_1x+\sigma_2x^2 \quad \mbox{and}\quad \int_{\R}z^m\ell(x,dz)=\sum_{i=0}^m\xi_i^mx^i.
\end{equation}
We consider an example. 
\begin{example}
\label{jumpdiff}
Let $B$ a standard one-dimensional Brownian motion and $\tilde{N}(dt, dz)$ a compensated Poisson random measure  with compensator $\nu(dz)dt$. We consider the jump-diffusion SDE given by
\begin{equation*}
	\label{sdepol}
	dY(t) = b(Y(t))dt + \sigma(Y(t))dB(t)+\int_{\R}\delta(Y(t^-),z)\tilde{N}(dt,dz),
\end{equation*} 
with drift, volatility and jump size functions of the form
\begin{equation*}
	\label{poldefcond2}
	b(x):=b_0+b_1x,\qquad
	\sigma^2(x) := \sigma_0 +\sigma_1x+\sigma_2x^2,\qquad
	\delta(x,z):=\delta_0(z)+\delta_1(z)x,
\end{equation*}
for $b_0, b_1, \sigma_0, \sigma_1, \sigma_2 \in \R$, and $\delta_0, \delta_1:\R \to \R$ such that $\int_{\R}|\delta_i(z)|^m\nu(dz)<\infty$ for all $m\ge 2$ and $i=0,1$. The SDE has a unique strong solution $Y(t)$ for each initial condition $Y(0)=y\in \R$. Moreover, $Y(t)$ is a polynomial jump-diffusion with linear drift $b\in \mathrm{Pol}_1(\R)$, quadratic diffusion $\sigma^2\in \mathrm{Pol}_2(\R)$, and jump measure $\ell(x, dz)$ given by $\int_{\R}f(z)\ell(x,dz)= \int_{\R} f(\delta(x,z))\nu(dz)$. In particular, for $m\ge2$, by the binomial theorem we find that
\begin{equation*}
\int_{\R}\left(\delta_0(z)+\delta_1(z)x\right)^m\nu(dz) = \sum_{i=0}^m\binom{m}{i}\int_{\R}\delta_0(z)^{m-i}\delta_1(z)^i\nu(dx)\;x^{i},
\end{equation*}
so that in this case the constants $\xi_i^m$ introduced in equation \eqref{poldefcond} are $\xi_i^m = \binom{m}{i}\int_{\R}\delta_0(z)^{m-i}\delta_1(z)^i\nu(dx)$, for $i=0, \dots, m.$
\end{example}

\subsection{The generator matrix}
\label{secgen} 
We consider the set $\{1,x,x^2, \dots, x^n\}$ as basis for $\mathrm{Pol}_n(\R)$, and we introduce the vector valued function 
\begin{equation*}
	\label{Hvec}
	H_n:\R \longrightarrow \R^{n+1}, \quad H_n(x) = (1,x,x^2,\dots,x^n)^{\top},
\end{equation*}
with $\top$ the transpose operator, so that every polynomial function $p\in \mathrm{Pol}_n(\R)$ with vector of coordinates $\vec{p} = (p_0,p_1,\dots,p_n)^{\top}\in\R^{n+1}$ can be represented by $p(x) = \vec{p}^{\top}H_n(x) = H_n(x)^{\top}\vec{p}$.

We report now rigorously the moment formula for polynomial processes from \cite[Theorem 2.5]{filipovic2020}, for which we include the proof that will be useful for the analysis in Section \ref{correlators}. 

\begin{theorem}[Moment formula]
	\label{genm}
	For $n\ge1$ and $Y$ polynomial process with extended generator $\G$:
	\begin{enumerate}	
		\item There exists a so-called \emph{generator matrix} $G_n \in \R^{(n+1)\times(n+1)}$ such that
		\begin{equation}
		\label{Gndef}
		\G H_n(x) = G_nH_n(x).
		\end{equation}		
		\item For every $p\in \mathrm{Pol}_n(\R)$ with vector of coefficients $\vec{p}\in\R^{n+1}$, the moment formula holds
		\begin{equation*}
		\E\left[p(Y(T))\left.\right|\F_t\right] =  \vec{p}^{\top}e^{G_n(T-t)}H_n(Y(t)), \quad 0\le t\le T.
		\end{equation*}
	\end{enumerate}
	\begin{proof}
		We take $f(x)=x^k$ for $0 \le k\le n$. Since $Y$ is a polynomial process, there exists $\vec{q}_k\in \R^{n+1}$ such that $\G x^k = \vec{q}_k^{\top}H_n(x)$, $0 \le k\le n$. With these vectors we can then construct a matrix $G_n \in \R^{(n+1)\times(n+1)}$ such that $\G H_n(x) = G_nH_n(x)$. This proves claim 1. 
		Next, by equation \eqref{martingale},  we write that 
		\begin{align}
		\E\left[p(Y(T))\left.\right|\F_t\right] = \vec{p}^{\top}\E[H_n(Y(T))\left.\right|\F_t] 	&= \vec{p}^{\top}H_n(Y(t))+\int_t^T\vec{p}^{\top}\E[\left.\G H_n(Y(s))\right|\F_t]ds \notag\\&= \vec{p}^{\top}H_n(Y(t))+\vec{p}^{\top}G_n\int_t^T\E[\left.H_n(Y(s))\right|\F_t]ds.\label{exp}
		\end{align}
		 We focus on $H_n(Y(T))$. For $Z(s):=\E[H_n(Y(s))\left.\right|\F_t]$, equation \eqref{exp} can be written in differential form as $dZ(s) = G_nZ(s)ds,$ whose solution, by separation of variables, is $Z(T) = e^{G_n(T-t)}Z(t)$. From the definition of $Z$, multiplying by the vector $\vec{p}^{\top}$, we conclude the proof.		
	\end{proof}
\end{theorem}

Theorem \ref{genm} tells us that $\E\left[p(Y(T))\left.\right|\F_t\right]$ is a polynomial function in $Y(t)$ for every $p \in  \mathrm{Pol}_n(\R)$. We point out that this holds for every choice of the vector basis of polynomials, despite in this paper we focus on the vector basis of monomials, $H_n(x)$. Moreover, we stress the fact that the moment formula strongly relies on the existence of the generator matrix $G_n$ and on the martingale property of the process in equation \eqref{martingale}. These two elements will be the key for all our framework.

\begin{example}
	\label{G2ex}
	Let $n=2$. Then $H_2(x) = (1, x, x^2)^{\top}$ and $\G H_2(x) = (\G 1, \G x, \G x^2)^{\top}.$ In particular, from equations \eqref{genD} and \eqref{poldefcond}, we get that $\G 1 = 0$, $\G x = b_0 + b_1 x$ and
	\begin{equation*}
		\small
		\label{G2}
		\G x^2 = \left(\sigma_0+\xi_0^2\right) +  \left(\sigma_1 + 2b_0+\xi_1^2\right) x +\left(\sigma_2 +2b_1+\xi_2^2\right) x^2.
	\end{equation*}
	One finds that the generator matrix $G_2\in \R^{3\times 3}$ satisfying \eqref{Gndef} is then
	\begin{equation*}
		G_2 = \begin{pmatrix}
			0 & 0 & 0\\ 
			b_0  & b_1 & 0 \\
			\sigma_0+\xi_0^2 & \sigma_1 + 2b_0+\xi_1^2 & \sigma_2 +2b_1+\xi_2^2
		\end{pmatrix} .
	\end{equation*}
\end{example}

\section{Two-point correlators}
\label{correlators}
Aiming at solving the $(m+1)$-point correlators problem in equation \eqref{corrn}, we start the analysis for $m=1$ because the tools and ideas developed to solve this case are crucial to understand the framework that will be generalized to $m+1$ polynomials in Section \ref{hcorr}. For $m=1$, equation \eqref{corrn} reads like
\begin{equation*}
C_{p_0,p_1}(s_0,s_1;t):=\E\left[\left.p_1\left(Y(s_0)\right)p_0\left(Y(s_1)\right)\right|\F_t\right], \qquad 0\le t < s_0 < s_1,
\end{equation*}
with $p_0\in \mathrm{Pol}_{n_0}(\R)$ and $p_1\in \mathrm{Pol}_{n_1}(\R)$. In particular, for $n := \max\{n_0, n_1\}$, we can represent the two polynomial functions $p_0$ and $p_1$ respectively by $p_0(x)=\vec{p}_{0}^{\top}H_n(x)$ and $p_1(x)=\vec{p}_{1}^{\top}H_n(x)$. By the tower rule for $\F_t \subseteq \F_{s_0}$ and the moment formula in Theorem \ref{genm}, $C_{p_0,p_1}(s_0,s_1;t)$ can be rewritten by
\begin{equation}
\label{corr2}
C_{p_0,p_1}(s_0,s_1;t)= \vec{p}_{1}^{\top} \E\left[\left.H_{n}(Y(s_0))H_{n}(Y(s_0))^{\top}\right|\F_t\right]e^{G_{n}^{\top}(s_1-s_0)}\vec{p}_{0}.
\end{equation}
This means that the conditional expectation of the product of two polynomial functions reduces to the conditional expectation of the outer product of the basis function $H_n(x)$ with itself, which  is a matrix of monomial functions of the form 
\begin{equation}
	\label{Xn2}
	X_n(x):=\NiceMatrixOptions{transparent}
	H_{n}(x)H_{n}(x)^{\top}=\begin{pmatrix}
		1 & x & x^2 & \cdots & x^n\\
		x & x^2 & x^3 & \cdots & x^{n+1}\\
		x^2 & x^3 & x^4 & \cdots & x^{n+2}\\
		\vdots &\vdots &\vdots & \ddots & \vdots \\
		x^n & x^{n+1} & x^{n+2} & \cdots & x^{2n}
	\end{pmatrix}.
\end{equation} 
By equation \eqref{martingale} we notice that
\begin{equation}
	\label{ode}
	\E\left[\left.X_n(Y(s_0))\right|\F_t\right] = X_n(Y(t))+ \int_t^{s_0} \E\left[\left.\G X_n(Y(s))\right|\F_t\right] ds.
\end{equation} 
Thus, in the same spirit of the proof of Theorem \ref{genm}, we seek a linear operator such that
\begin{equation}
	\label{lo}
	G_{n}^{(1)}:\R^{(n+1)\times(n+1)}\longrightarrow \R^{(n+1)\times(n+1)}, \qquad \G X_n(x) = G_{n}^{(1)}X_n(x),
\end{equation}
which is the equivalent linear operator in the two-polynomial setting to the generator matrix $G_n$. However, $G_{n}^{(1)}$ cannot be represented with a matrix. We notice that $G_n$ maps a vector to a vector, while $G_n^{(1)}$ maps a matrix to a matrix. The idea is then to transform the matrix-matrix problem into a vector-vector problem and to construct the linear operator $G_n^{(1)}$ in terms of the generator matrix $G_n$. We start by introducing the following operators for a general matrix $A \in \R^{n\times m}$.
\begin{definition}[Vectorization and inverse-vectorization]
	\label{vecD}
	Given a matrix $A \in \R^{n\times m}$ whose $j$-th column we denote by $A_{:j}$, we define $vec: \R^{n\times m} \to \R^{nm}$ as the operator that associates to $A$ the $nm$-column vector
	\begin{equation*}
		\NiceMatrixOptions{transparent}
		vec(A) = \begin{pmatrix}
			A_{:1}^{\top} & A_{:2}^{\top}&\cdots & A_{:m}^{\top}
		\end{pmatrix}^{\top},
	\end{equation*}
	which is called the \emph{vectorization} of $A$. For $v=vec(A)$, we then define $vec^{-1}: \R^{nm} \to \R^{n\times  m}$ as the operator that associates to the vector $v$ the $n\times m$ matrix
	$B = vec^{-1}(v)$, such that $[B]_{i,j} = v_{n(j-1)+i}$, for $i=1, \dots, n$ and $j=1, \dots, m$. In this case, we say that $B$ is the \emph{inverse-vectorization} of $v$. In particular, $B$ and $A$ coincide.
\end{definition}

We then address the problem of finding the linear operator $G_{n}^{(1)}$ transforming $X_n(x)$ into $\G X_n(x)$, to the problem of finding a matrix $\tilde{G}_n^{(1)} \in \R^{(n+1)^2\times (n+1)^2}$ such that
\begin{equation}
	\label{Mn}
	\G vec(X_n(x)) =  \tilde{G}_n^{(1)} vec(X_n(x)),
\end{equation}
where $\G vec(X_n(x)) =  vec(\G X_n(x))$. The operator $G_{n}^{(1)}$ satisfying equation  \eqref{lo} is then obtained by composing the matrix $\tilde{G}_n^{(1)}$ with the $vec$ and $vec^{-1}$ operators, namely
\begin{equation}
\label{Gn2}
G_{n}^{(1)} = vec^{-1} \circ \tilde{G}_{n}^{(1)} \circ vec.
\end{equation}
We consider an example. 

\begin{example}
		\label{M1}
		Let $n=1$. We seek $\tilde{G}_{1}^{(1)} \in \R^{4\times 4}$ such that $\left(\G 1, \G x , \G x , \G x^2\right)^{\top} = \tilde{G}_{1}^{(1)} \left(1, x , x , x^2\right)^{\top}$. 
		Two suitable choices of $\tilde{G}_{1}^{(1)} $ are
		\begin{align*}
			\footnotesize
			&\tilde{G}_{1}^{(1)} = \begin{pmatrix}
				0 & 0 & 0 & 0 \\
				b_0 & b_1 & 0 & 0 \\
				b_0 & b_1 & 0 & 0 \\
				\sigma_0+\xi_0^2 & \sigma_1 + 2b_0+\xi_1^2 & 0 & \sigma_2 +2b_1+\xi_2^2
			\end{pmatrix} \quad \mbox{\normalsize and}\\
			&\tilde{G}_{1}^{(1)} = \begin{pmatrix}
				0 & 0 & 0 & 0 \\
				b_0 & b_1/2 & b_1/2  & 0 \\
				b_0 & b_1/2 & b_1/2  & 0 \\
				\sigma_0+\xi_0^2 & \left(\sigma_1 + 2b_0+\xi_1^2\right)/2 & \left(\sigma_1 + 2b_0+\xi_1^2\right)/2 & \sigma_2 +2b_1+\xi_2^2
			\end{pmatrix}.
		\end{align*}
\end{example}

We notice from Example \ref{M1} that the first $\tilde{G}_{1}^{(1)}$ has two identical rows and a null column, while the second $\tilde{G}_{1}^{(1)}$ has both two identical rows and two identical columns. This is due to the double presence of the term $\G x$ in $vec(\G X_1(x))$, or, analogously, the double presence of the term $x$ in $vec(X_1(x))$. Increasing the value of $n$, the number of redundant terms in $vec(\G X_n(x))$ and $vec(X_n(x))$ increases, hence to find a recursion for the matrix $\tilde{G}_{n}^{(1)}$ seems not an easy task. Moreover, we would like to write the matrix $\tilde{G}_{n}^{(1)}$ in terms of the generator matrix $G_n$. We shall solve this issue in the next section.

\subsection{The L-vectorization}
\label{LL}
Looking at equation \eqref{Xn2}, we notice that a possible way, among others, to get from the matrix $X_n(x)$ all the elements without repetition (that is equivalent to get all the powers of $x$ from $0$ to $2n$ without repetition) is to select the first column and the last row. For this, we introduce the following operator.

\begin{definition}[L-vectorization]
\label{vecL}
Given a matrix $A \in \R^{n\times m}$ with elements $[A]_{i,j} = a_{i,j}$, $1\le i \le n$ and $1 \le j \le m$, we define the \emph{L-vectorization} of $A$ as the operator $vecL: \R^{n\times m} \to \R^{n+m-1}$ that associates to $A$ the $(n+m-1)$-column vector obtained by selecting the first column and the last row of $A$, namely
\begin{equation*}
\NiceMatrixOptions{transparent}
vecL(A) = \begin{pmatrix}
	a_{1,1}&a_{2,1}&\cdots&a_{n,1}&a_{n,2}&\cdots&a_{n,m}
\end{pmatrix}^{\top}.
\end{equation*}
\end{definition}
Intuitively, the $vecL$ operator is a linear operator selecting from the matrix $A$ the elements that together form the biggest "L" inscribed in the matrix $A$.  In \cite{magnus1980}, the authors introduce the half-vectorization operator, which, starting from a matrix $A$, returns the column vector obtained by stacking together the columns of the lower-triangular matrix contained in $A$. Moreover, they provide two matrices, the eliminating matrix and the duplicating matrix, that, respectively, transform the vectorization of $A$ into the half-vectorization, and vice-versa. We aim at the same kind of results for the L-vectorization. The existence of such matrices tells us that there exist a linear transformation to remove the duplicates from $vec(X_n(x))$ (what we call the \emph{L-eliminating matrix}) and the corresponding inverse linear transformation (the \emph{L-duplicating matrix}).

From now on, we shall denote with $\vec{e}_{k,j}$ the $j$-th canonical basis vector in $\R^k$, with $I_k$ the identity matrix in $\R^{k\times k}$, and with $\otimes$ the Kronecker product, for which we recall the definition.
\begin{definition}[Kronecker product]
	\label{kronD}
	The \emph{Kronecker product} of a matrix $A \in \R^{n\times m}$ with elements $[A]_{i,j} = a_{i,j}$, $1\le i \le n$ and $1 \le j \le m$, and a matrix $B \in \R^{r\times s}$, is the matrix $A \otimes B \in \R^{nr \times ms}$ given by
	\begin{equation*}
		\NiceMatrixOptions{transparent}
		A \otimes B = \begin{pmatrix}
			a_{1,1}B& \cdots &a_{1,m}B\\
			\vdots & \ddots& \vdots \\
			a_{n,1}B & \cdots & a_{n,m}B
		\end{pmatrix}.
	\end{equation*}
\end{definition}

We define now the L-eliminating matrix.

\begin{theorem}[L-eliminating matrix]
	\label{Enmth}
	For every $n,m\ge 1$ and for every matrix $A \in \R^{n\times m}$, there exists an L-eliminating matrix $E_{n,m}\in \R^{(n+m-1)\times nm}$ such that
	\begin{align}
		\label{Enm2} 
		&E_{n,m}vec(A) = vecL(A)\\
		&E_{n,m} = \sum_{i=1}^n \vec{e}_{n+m-1,i}\otimes \vec{e}_{m,1}^{\top}\otimes \vec{e}_{n,i}^{\top}+\sum_{i=2}^m \vec{e}_{n+m-1,n+i-1}\otimes\vec{e}_{m,i}^{\top}\otimes \vec{e}_{n,n}^{\top}.\label{Enmex2}
	\end{align}
\end{theorem}

\begin{corollary}
	\label{cor11}
	For every $n\ge1$, the L-eliminating matrix $E_{n+1} \in \R^{(2n+1)\times(n+1)^2}$ transforming $vec(X_n(x))$ into $vecL(X_n(x))$ is given by $E_{n+1} := E_{n+1,n+1}$.
\end{corollary} 

\begin{example}
	\label{E1ex2}
	Let $n=m=2$. Then equation \eqref{Enmex2} becomes
	\begin{equation*}
		E_{2,2} = \sum_{i=1}^{2} \vec{e}_{3,i}\otimes \vec{e}_{2,1}^{\top}\otimes \vec{e}_{2,i}^{\top}+\sum_{i=2}^{2} \vec{e}_{3,2+i}\otimes\vec{e}_{2,i}^{\top}\otimes \vec{e}_{2,2}^{\top}=\begin{pmatrix}
			1& 0 & 0 & 0\\
			0&1& 0 & 0 \\
			0& 0 & 0 & 1
		\end{pmatrix}.
	\end{equation*}
	For $A\in \R^{2\times 2}$ of the form $A = \begin{pmatrix}
		a_1 & a_3 \\
		a_2 & a_4
	\end{pmatrix}$, 
	we can verify that $E_{2,2}$ satisfies the definition of an L-eliminating matrix in equation \eqref{Enm2}, as indeed
	\begin{equation*}
		E_{2,2} vec(A) = 
		\begin{pmatrix}
			1& 0 & 0 & 0\\
			0&1& 0 & 0 \\
			0& 0 & 0 & 1
		\end{pmatrix}
		\begin{pmatrix}
			a_1 \\
			a_2 \\
			a_3\\
			a_4
		\end{pmatrix} 
		= \begin{pmatrix}
			a_1 \\
			a_2 \\
			a_4
		\end{pmatrix} = vecL(A).
	\end{equation*}
	Moreover, when applied to $vec(X_1(x))=(1, x,x, x^2)^{\top}$ it eliminates the double value $x$.
\end{example}

Next, we want to define an inverse operator to $E_{n,m}$, namely a linear mapping transforming the L-vectorization of a matrix $A$ into its vectorization. However, this inverse operation is not well defined in the space of matrices in $\R^{n\times m}$. Indeed, when applying $E_{n,m}$ to $vec(A)$, we go from a space of dimension $nm$ to a space of dimension $n+m-1 < nm$. Then the inverse transformation in general does not exist. Thus, it is necessary to find a suitable subspace of $\R^{n\times m}$ of dimension of  $n+m-1$, so that image space dimension and domain space dimension coincide. In \cite{magnus1980}, the authors face a similar issue which they solve by restricting the domain to the space of symmetric matrices. 

Looking at the matrix of functions $X_n(x)$, we notice that each ascending skew-diagonal from left to right is constant, which is a property of the so-called \emph{Hankel matrices}. This class of matrices is usually defined in the square case; we however consider an extended definition to rectangular matrices as introduced, for example, in \cite{fiedler1985}.
\begin{definition}[Hankel matrix]
	\label{spaceA}
	We define $\mathcal{A}_{n,m}\subset \R^{n\times m}$ as the space of matrices whose elements on the same skew-diagonal coincide. We distinguish two cases corresponding to whether $n\ge m$ or $m\ge n$, so that a matrix $A \in\mathcal{A}_{n,m}$ takes one the following two forms:
	\begin{equation*}
		\footnotesize
		\NiceMatrixOptions{transparent}
		\label{Anm2}
		A = \begin{pmatrix}
			a_1       & 	a_2 &    \cdots      	 & a_m \\
			a_2       &  		  &\iddots   &  \vdots     \\
			\vdots  &\iddots&               &  \vdots     \\
			a_m     &             &              & a_n \\
			\vdots &            &\iddots &  \vdots     \\
			\vdots &\iddots&              & \vdots      \\
			a_n     & \cdots&  \cdots& a_{n+m-1}
		\end{pmatrix}\; \quad 
		\mbox{\normalsize or } \quad \; A = \begin{pmatrix}
			a_1       & 	a_2 &    \cdots  & a_n   & \cdots & \cdots & a_m \\
			a_2       &  		  &\iddots      &        &            & \iddots& \vdots \\
			\vdots  &\iddots&                 &        & \iddots&           &\vdots \\
			a_n     & \cdots & \cdots      & a_m  &  \cdots & \cdots &  a_{n+m-1} \\
		\end{pmatrix},
	\end{equation*}
	for $a_1, \dots, a_{n+m-1}\in \R$. We call $A \in \mathcal{A}_{n,m}$ an  \emph{Hankel matrix} and write $\mathcal{A}_{n} := \mathcal{A}_{n,n}$ for $n=m$.
\end{definition}

We see that $X_n(x)\in\mathcal{A}_{n+1}$, and we can now define the inverse operator of $E_{n,m}$ on $\mathcal{A}_{n,m}$.
\begin{theorem}[L-duplicating matrix]
	\label{Dnth}
	For every $n,m\ge1$ and for every matrix $A \in \mathcal{A}_{n,m}$, there exists an L-duplicating matrix $D_{n,m}\in\R^{nm\times (n+m-1)}$ such that
	\begin{align}
		&\label{Dn2} 
		D_{n,m}vecL(A) = vec(A)\\
		&\label{Dnex2}
		D_{n,m} = \sum_{i=1}^{n}\sum_{j=1}^{m} \vec{e}_{n+m-1,i+j-1}^{\top}\otimes\vec{e}_{m,j}\otimes \vec{e}_{n,i}.
	\end{align}
\end{theorem}

\begin{corollary}
	\label{cor22}
	For every $n\ge1$, the L-duplicating matrix $D_{n+1} \in \R^{(n+1)^2\times (2n+1)}$ transforming $vecL(X_n(x))$ into $vec(X_n(x))$ is given by $D_{n+1}:=D_{n+1,n+1}$. 
\end{corollary} 

\begin{example}
	\label{D1ex2}
	Let $n=m=2$. Then equation \eqref{Dnex2} becomes
	\begin{equation*}
		D_{2,2} = \sum_{i=1}^{2}\sum_{j=1}^{2} \vec{e}_{3,i+j-1}^{\top}\otimes\vec{e}_{2,j}\otimes \vec{e}_{2,i}=
		\begin{pmatrix}
			1& 0 &  0\\
			0&1 & 0 \\
			0&1 & 0 \\
			0& 0 & 1
		\end{pmatrix}.
	\end{equation*}
	For $A\in \mathcal{A}_{2}$ of the form $A=\begin{pmatrix}
		a_1 & a_2\\ a_2& a_3\end{pmatrix}$, we can verify that
	\begin{equation*}
		D_{2,2}  vecL(A) = 
		\begin{pmatrix}
			1& 0 &  0\\
			0&1 & 0 \\
			0&1 & 0 \\
			0& 0 & 1
		\end{pmatrix}
		\begin{pmatrix}
			a_1 \\
			a_2 \\
			a_3
		\end{pmatrix} 
		= \begin{pmatrix}
			a_1 \\
			a_2 \\
			a_2\\
			a_3
		\end{pmatrix} = vec(A).
	\end{equation*}
	Moreover, when applied to $vecL(X_1(x))=(1, x,x^2)^{\top}$, it duplicates the missing value $x$.
\end{example}

We conclude this section with an important property for the matrices $E_{n,m}$ and $D_{n,m}$.
\begin{proposition}
	\label{EDp}
	For every $n,m\ge 1$, $D_{n,m}$ is the right-inverse of $E_{n,m}$ and for every $A\in \mathcal{A}_{n,m}$, the product $D_{n,m} E_{n,m}\in\R^{nm\times nm}$ acts on $vec(A)$ like an identity operator, $D_{n,m}E_{n,m} vec(A)=vec(A)$.
\end{proposition}

\begin{example}
\label{counter}
Let  $n=m=2$. From Example \ref{E1ex2} and \ref{D1ex2} we get
\begin{equation*}
D_{2,2}E_{2,2} =\begin{pmatrix}
	1& 0 &  0\\
	0&1 & 0 \\
	0&1 & 0 \\
	0& 0 & 1
\end{pmatrix}
\begin{pmatrix}
	1& 0 & 0 & 0\\
	0&1& 0 & 0 \\
	0& 0 & 0 & 1
\end{pmatrix} = 
\begin{pmatrix}
	1& 0 &  0& 0 \\
	0&1 & 0 & 0\\
	0&1 & 0 & 0\\
	0& 0 & 0& 1
\end{pmatrix} \ne I_4.
\end{equation*}
However, for $A\in \mathcal{A}_{2}$ with $vec(A)=(a_1, a_2, a_2, a_3)^{\top}$,  we notice that $D_{2,2}E_{2,2}vec(A)=vec(A)$, hence the product $D_{2,2}E_{2,2}$ behaves like an identity operator, despite not coinciding with the identity matrix.
\end{example}

\subsection{The generator for correlators}
We focus now on the original problem: by equation \eqref{Gn2} we seek a linear operator $\tilde{G}_n^{(1)}$ transforming $vec(X_n(x))$ into $\G vec(X_n(x))$. From equation \eqref{Xn2}, we notice that the elements of $X_n(x)$ lying on the left-bottom "L" coincide with the powers of $x$ from $0$ to $2n$. 

\begin{lemma}
	\label{id}
	For every $n\ge 1$, the L-vectorization of $X_n(x)$ coincides with the vectors basis of monomials of order $2n$, namely $vecL(X_n(x))= H_{2n}(x).$
\end{lemma}

Hence, by transforming $vec(X_n(x))$ into $vecL(X_n(x))$, we address the problem of finding the generator matrix for $vec(X_n(x))$ to the problem of finding the generator matrix for $H_{2n}(x)$, which was solved in Section \ref{secgen}. We can then prove the following result.

\begin{proposition}
	\label{Mnex}
	For every $t\ge 0$ and $n\ge 1$, the matrix $\tilde{G}_n^{(1)}$ satisfying equation \eqref{Mn} and its matrix exponential are respectively given by
	\begin{equation*}
		\tilde{G}_n^{(1)} = D_{n+1} G_{2n} E_{n+1} \qquad \mbox{and} \qquad 	e^{\tilde{G}_n^{(1)} t} =  D_{n+1} e^{G_{2n}t} E_{n+1} ,
	\end{equation*}
	where $G_{2n}$ is the generator matrix of order $2n$.
\end{proposition}

We are now able to provide a solution to the two-point correlator problem.
\begin{theorem}[Two-point correlator formula]
	\label{odesol2}
	The expression for the two-point correlator is given by
	\begin{equation*}
		C_{p_0,p_1}(s_0,s_1;t)= \vec{p}_{1}^{\top} \left\{vec^{-1} \circ D_{n+1} e^{G_{2n} (s_0-t)} E_{n+1}\circ vec \left(X_{n}(Y(t))\right)\right\}e^{G_{n}^{\top}(s_1-s_0)}\vec{p}_{0},
	\end{equation*}
with $\vec{p}_{0}, \vec{p}_{1} \in \R^{n+1}$ the vectors of coefficients for the polynomial functions $p_0\in \mathrm{Pol}_{n_0}(\R)$ and $p_1\in \mathrm{Pol}_{n_1}(\R)$, $n=\max\{n_0, n_1\}$ and $t< s_0 < s_1$. 
\end{theorem}

\section{Higher-order correlators}
\label{hcorr}
We prove in this section a correlator formula holding for every $m\ge 1$ by following similar steps to the ones performed in Section \ref{correlators} for $m=1$. We recall that we seek an explicit expression for
\begin{equation*}
	C_{p_0,\dots, p_m}(s_0,\dots, s_m;t):= \E\left[p_m\left(Y(s_0)\right)p_{m-1}\left(Y(s_1)\right)\cdot \dots \cdot p_0\left(Y(s_m)\right)\left.\right|\F_t\right],
\end{equation*}
with $p_k\in \mathrm{Pol}_{n_k}(\R)$, $k=0,\dots,m$, and $t < s_0 < s_1 < \dots < s_m$. We start with the following operator.
\begin{definition}[d-Kronecker product]
	\label{mkrondef2}
	We define the \emph{d-Kronecker product} of a matrix $A \in \R^{n\times m}$ and a matrix $B \in \R^{r\times s}$, as the $d$-th Kronecker power of $A$ multiplied in the Kronecker sense with $B$, for $d\ge1$, or equal to $B$, for $d=0$, namely
	\begin{equation*}
		\label{mkron}
		\begin{cases}
			A\otimes ^{d} B = A^{\otimes d} \otimes B & d\ge 1\\
			A\otimes ^{0} B = B & d = 0\\
		\end{cases}.
	\end{equation*}
\end{definition}
\noindent Then for $n\ge 1$ and $r\ge 0$, we introduce the matrix of functions
\begin{equation}
	\label{Xnm}
	X_n^{(r)}(x):=H_n(x)^{\top} \otimes^rH_n(x),
\end{equation}
for which we can make the following considerations:
\begin{itemize}[leftmargin=*]
	\item for $r=0$: we get $X_n^{(0)}(x) = H_n(x)$;
	\item for $r=1$: we get 
	\begin{equation}
		\label{Xn1}
		X_n^{(1)}(x) = H_n(x)^{\top} \otimes H_n(x)= H_n(x)H_n(x)^{\top} = X_n(x) \in \mathcal{A}_{n+1};
	\end{equation}
	\item for $r=2$: by the associativity property of the Kronecker product
	\begin{equation*}
		X_n^{(2)}(x) =   H_n(x)^{\top} \otimes X_n^{(1)}(x)=
		\left(\begin{NiceArray}{c:c:c:c:c}
			X_n^{(1)}(x) & x X_n^{(1)}(x) & x^2 X_n^{(1)}(x) & \cdots\cdots & x^n X_n^{(1)}(x)
		\end{NiceArray}\right)
	\end{equation*}
	is composed by $n+1$ blocks of the form $B_{n,2}^{(k)}=x^{k-1} X_n^{(1)}(x)\in \mathcal{A}_{n+1}$, $k=1,\dots,(n+1)$;
	\item for $r=3$: we write $X_n^{(3)}(x) =  \left(H_n(x)^{\top}\right)^{\otimes 2} \otimes X_n^{(1)}(x)$, where
	\begin{gather*}
		\left(H_n(x)^{\top}\right)^{\otimes 2}=\left(\begin{NiceArray}{cccc:cccc:c:cccc}
			1& x & \Cdots & x^n & x & x^2 & \Cdots & x^{n+1}& \cdots & x^n & x^{n+1} &  \Cdots & x^{2n}
		\end{NiceArray}\right) \in \R^{(n+1)^2}
	\end{gather*}
	so that $X_n^{(3)}(x)$ is composed by $(n+1)^2$ blocks. For each block $B_{n,3}^{(k)}$ there exists $j_k \in \{0,\dots, 2n\}$ such that $B_{n,3}^{(k)}=x^{j_k} X_n^{(1)}(x)$ and $B_{n,3}^{(k)} \in \mathcal{A}_{n+1}$, $k=1, \dots, (n+1)^2$. The difference from the previous case is that, now, some of the blocks are repeated since in $\left(H_n(x)^{\top}\right)^{\otimes 2}$ some powers are repeated.
\end{itemize}

Generalizing, we can state the following result.

\begin{proposition}
	\label{shapeX}
	For every $n,r\ge 1$, $X_n^{(r)}(x)$  is a rectangular block matrix in $\R^{(n+1)\times(n+1)^r}$, which is composed by $(n+1)^{r-1}$ blocks $B_{n,r}^{(k)}(x) \in \mathcal{A}_{n+1}$, for which there exists an index $j_k \in \{0, \dots, (r-1)n\}$ such that $B_{n,r}^{(k)}(x)= x^{j_k}X_n^{(1)}(x)$, $k=1,\dots,(n+1)^{r-1}$.
\end{proposition}

One can also notice that $X_n^{(r)}(x)$ contains all the powers of $x$ from $0$ to $(r+1)n$. Thus, after removing the redundant powers, we are left with $H_{n(r+1)}(x)$. However, from Proposition \ref{shapeX}, $X_n^{(r)}(x)$ is a block matrix whose blocks belong to $\mathcal{A}_{n+1}$, but the matrix itself does not belong to $\mathcal{A}_{n+1,(n+1)^{r}}$. 
\begin{example}
	\label{noA}
	Let $n=r=2$. Then we get the following matrix
	\begin{equation*}
		X_2^{(2)}(x) =
		\left(\begin{array}{ccc:ccc:ccc}
			1 & x &x^2& x & x^2&x^3&x^2&x^3 & x^4 \\
			x & x^2 &x^3& x^2 & x^3&x^4&x^3&x^4 & x^5 \\
			x^2 & x^3 &x^4& x^3 & x^4&x^5&x^4&x^5 & x^6 \\
		\end{array}\right)
	\end{equation*}
	whose blocks belong to $\mathcal{A}_{3}$, but $X_2^{(2)}(x)\notin\mathcal{A}_{3,9}$ hence it is not a Hankel matrix.
\end{example}

This means, in particular, that we cannot use the L-eliminating and L-duplicating matrices as defined in Section \ref{correlators}. We need instead two new tailor-made operators.

\begin{proposition}
	\label{propEp}
	For $n, m\ge 1$, there exists an \emph{$m$-th L-eliminating matrix} $E_{n+1}^{(m)}\in \R^{(n(m+1)+1)\times(n+1)^{m+1}}$ such that $E_{n+1}^{(m)}vec(X_n^{(m)}(x)) = H_{n(m+1)}(x)$. In particular, $E_{n+1}^{(m)}$ is given by the recursion formula
	\begin{equation*}
		\label{recE}
		\begin{cases}
			E_{n+1}^{(1)} = E_{n+1}& m=1\\
			E_{n+1}^{(m)} = E_{nm+1, n+1}\left(I_{n+1} \otimes E_{n+1}^{(m-1)}\right) & m\ge 2
		\end{cases}.
	\end{equation*}
\end{proposition}

\begin{proposition}
	\label{propDp}
	For $n,m\ge 1$, there exists an \emph{$m$-th L-duplicating matrix} $D_{n+1}^{(m)}\in \R^{(n+1)^{m+1}\times (n(m+1)+1)}$ such that $D_{n+1}^{(m)} H_{n(m+1)}(x)=vec(X_n^{(m)}(x))$. In particular, $D_{n+1}^{(m)}$ is given by the recursion formula
	\begin{equation*}
		\label{recD}
		\begin{cases}
			D_{n+1}^{(1)} = D_{n+1}& m=1\\
			D_{n+1}^{(m)} = \left(I_{n+1} \otimes D_{n+1}^{(m-1)}\right)D_{nm+1, n+1} & m\ge 2
		\end{cases}.
	\end{equation*}
\end{proposition}

\begin{proposition}
	\label{identity}
	For every $n,m\ge 1$, the matrix $D_{n+1}^{(m)}$ is the right-inverse of $E_{n+1}^{(m)}$, and the product $D_{n+1}^{(m)}E_{n+1}^{(m)}$ acts on $vec(X_n^{(m)}(x))$ like an identity operator, $D_{n+1}^{(m)}E_{n+1}^{(m)}vec(X_n^{(m)}(x))=vec(X_n^{(m)}(x))$.
\end{proposition}

\begin{example}
	\label{ex}
	Let $n=m = 2$. By Proposition \ref{propEp}, we find that $E_3^{(1)}= E_{3}$ and
	\begin{equation*}
		\footnotesize
E_3^{(2)}= E_{5,3}\left(I_{3} \otimes E_3^{(1)} \right) =E_{5,3}
		\left(\begin{array}{c:c:c}%
			E_3^{(1)} & \bigzero & \bigzero \\
			\hdashline
			\bigzero & E_3^{(1)} & \bigzero \\
			\hdashline
			\bigzero & \bigzero & E_3^{(1)} 
		\end{array}\right).
	\end{equation*} 
	To understand better, we  write $vec(X_2^{(2)}(x))$ in Example \ref{noA} as follows:
	\begin{equation*}
		\footnotesize
		vec(X_2^{(2)}(x)) = 
		vec\left(\begin{array}{ccc}
			vec\begin{pmatrix}
				\text{\textcircled{$1$}} & x &x^2\\
				\text{\textcircled{$x$}} & x^2 &x^3\\
				\text{\textcircled{$x^2$}} & \text{\textcircled{$x^3$}} &\text{\textcircled{$x^4$}}
			\end{pmatrix}\,
			vec\begin{pmatrix}
				\text{\textcircled{$x$}}  & x^2 &x^3\\
				\text{\textcircled{$x^2$}}  & x^3 &x^4\\
				\text{\textcircled{$x^3$}}  & \text{\textcircled{$x^4$}}  &\text{\textcircled{$x^5$}} 
			\end{pmatrix}\,
			vec \begin{pmatrix}
				\text{\textcircled{$x^2$}}  & x^3 &x^4\\
				\text{\textcircled{$x^3$}}  & x^4 &x^5\\
				\text{\textcircled{$x^4$}}  & \text{\textcircled{$x^5$}}  &\text{\textcircled{$x^6$}} 
			\end{pmatrix}
		\end{array}\right),
	\end{equation*}
	so that, by applying $I_{3} \otimes E_3^{(1)}$, we select from each of the three blocks of $X_2^{(2)}(x)$ their L-vectorization (remember that the L-eliminating matrix acts on the vectorization of a matrix and returns the L-vectorization of the matrix itself), elements which are marked with a circle. We are left with 
	\begin{equation}
		\footnotesize
		\label{mr}
		\left(I_3 \otimes E_3^{(1)}\right)vec(X_2^{(2)}(x)) = 
		vec\left(\begin{array}{ccc}
			\text{\textcircled{$1$}}  & x & x^2 \\
			\text{\textcircled{$x$}} & x^2 & x^3\\
			\text{\textcircled{$x^2$}}  & x^3 & x^4\\
			\text{\textcircled{$x^3$}}  & x^4 & x^5\\
			\text{\textcircled{$x^4$}}  & \text{\textcircled{$x^5$}}  & \text{\textcircled{$x^6$}} \\
		\end{array}\right).
	\end{equation}
	We notice that the elements we need are on the left-bottom "L" (the ones marked with a circle). Applying $E_{5,3}$ gives $H_6(x)= H_{2(2+1)}(x)$. Moreover, the matrix on the right hand side of equation \eqref{mr} belongs to $\mathcal{A}_{5,3}$. Then the corresponding L-duplicating matrix is given by Proposition \ref{propDp} and is
	\begin{equation*}
\footnotesize
D_3^{(2)}= \left(I_{3} \otimes D_3^{(1)} \right) D_{5,3}=
		\left(\begin{array}{c:c:c}%
			D_3^{(1)} & \bigzero & \bigzero \\
			\hdashline
			\bigzero & D_3^{(1)} & \bigzero \\
			\hdashline
			\bigzero & \bigzero & D_3^{(1)} 
		\end{array}\right)D_{5,3}.
	\end{equation*} 
	In particular, $D_{5,3}$ acting on $H_6(x)$ returns the matrix on the right hand side of equation \eqref{mr} while $D_3^{(1)}$, acting singularly on each column because of the multiplication with $I_{3}$ (namely, $I_{3} \otimes D_3^{(1)}$) returns $X_2^{(2)}(x)$, showing that $D_3^{(2)}$ is the inverse operator of $E_3^{(2)}$.
\end{example}

We now derive the closed formula for the $(m+1)$-point correlators.

\begin{theorem}[Correlator formula]
\label{theoremformula}
For every $m\ge1$, let $p_k\in \mathrm{Pol}_{n_k}(\R)$ with vector of coefficients $\vec{p}_{k}\in \R^{n+1}$, $k=0,\dots,m$, $n= \max\{n_0,\dots,n_m\}$ and $t < s_0 < s_1 < \dots < s_m$. There exist $m+1$ matrices $\tilde{G}_n^{(r)}\in \R^{(n+1)^{r+1}\times(n+1)^{r+1}}$, $r=0, \dots, m$, such that
\begin{multline*}
C_{p_0,\dots, p_m}(s_0,\dots, s_m;t)=\\ \vec{p}_{m}^{\top}\left\{ vec^{-1} \circ e^{\tilde{G}_n^{(m)}(s_0-t)}\circ vec\left(X_n^{(m)}(Y(t))\right)\right\} \prod_{k=1}^{m}e^{\tilde{G}_n^{(m-k)\top}(s_k-s_{k-1})}\left\{I_{n+1}\otimes^{m-k}\vec{p}_{m-k} \right\}
\end{multline*}
where $\prod_{k=1}^{m}$ is the product obtained starting with the matrix corresponding to $k = 1$ and multiplying on the right by the following matrices until the matrix corresponding to $k = m$. In particular, $\tilde{G}_n^{(r)} = D_{n+1}^{(r)}G_{n(r+1)}E_{n+1}^{(r)}$
 and $e^{\tilde{G}_n^{(r)}t} = D_{n+1}^{(r)}e^{G_{n(r+1)}t}E_{n+1}^{(r)}$, with $\tilde{G}_n^{(0)}= G_n^{(0)}= G_n$.
\end{theorem}

\section{Recursion formulas for the generator matrix}
\label{recursionG}
We focus in this section on the generator matrix $G_n$ defined in Theorem \ref{genm}. In particular, we provide a recursion formula for it and a second recursion for its matrix exponential. These formulas are referred to the basis vector of monomials, but they can also be  generalized to a different polynomial basis vector. In this case, one needs the matrix for the change of basis.

\subsection{The generator matrix}
We provide a recursion formula for the generator matrix.
\begin{theorem}[Generator matrix recursion]
	\label{Gn}
	For every $n\ge 2$, the generator matrix $G_n\in \R^{(n+1)\times (n+1)}$ satisfying equation \eqref{Gndef} with respect to the vector basis of monomials $H_n(x)$ is given by
	\begin{equation*}
		G_n = \begin{pmatrix}
			G_{n-1} & \vec{0}_n \\ \vec{a}_n^{\top} & c_n
		\end{pmatrix} \quad \mbox{ with } \quad G_1 = \begin{pmatrix}
			0 & 0 \\ b_0 & b_1
		\end{pmatrix} .
	\end{equation*}
	Here $\vec{0}_n$ is a $n$-dimensional vector of $0$'s, $\vec{a}_n = (a_n^n, a_n^{n-1},\dots, a_n^1)^{\top} \in \R^n$ with
	\begin{equation}
		\label{an}
		\begin{aligned}
			a_n^1 &= nb_0+\frac{1}{2}n(n-1)\sigma_1+\sum_{k=2}^n \binom{n}{k} \xi_{k-1}^k, \qquad a_n^2 =\frac{1}{2}n(n-1)\sigma_0+\sum_{k=2}^n \binom{n}{k} \xi_{k-2}^k,\\
			a_n^i &= \sum_{k=i}^n \binom{n}{k}\xi_{k-i}^k \quad \mbox{ for } i = 3,\dots,n,\qquad \mbox{and }c_n = nb_1 + \frac{1}{2}n(n-1)\sigma_2 + \sum_{k=2}^n \binom{n}{k} \xi_k^{k}.
		\end{aligned}
	\end{equation}
\end{theorem}

\begin{remark}
	\label{oss}
	From Theorem \ref{Gn}, we notice that for $n\ge 1$ the generator matrix $G_n$ is lower triangular. Moreover, for $n\ge 2$ the main diagonal of $G_n$ is of the form
	\begin{equation}
		\label{diag}
		\mathrm{diag}\left(G_n\right) = \left(0, \,b_1,\, c_2,\,c_3,\, \dots, c_n\right)^{\top},
	\end{equation}
	so that, in particular, the matrix $G_n$ is not invertible. 
\end{remark}

\begin{lemma}
	\label{delta0}
	If $\ell(x,dz)\equiv0$ on $\R$, then $G_n$ is a (lower) tri--diagonal matrix.
\end{lemma}

We now provide a recursion formula for the matrix exponential of $G_n$.  
\begin{theorem}[Exponential generator matrix recursion]
\label{Gnk}
For a fixed $n\ge 2$, if the following conditions are satisfied
\begin{equation}
	\label{lambdacond}
	\begin{cases}
		c_j \ne 0 & \mbox{for every }\; 2 \le j \le n\\
		c_j \ne c_i &\mbox{for every }\; 1 \le j  < i\le n
	\end{cases}
\end{equation}
then the recursion formula holds:
\begin{align*}
&e^{G_nt} =\begin{pmatrix}
e^{G_{n-1}t} & \vec{0}_n \\
\vec{a}_n^{\top}\Lambda_n^{-1}\left(e^{c_nt}I_{n}-e^{G_{n-1}t}\right) & e^{c_nt} 
\end{pmatrix} \quad \mbox{ with } \\&e^{G_1t} =\begin{pmatrix}
1 & 0 \\
\frac{b_0}{b_1}\left(e^{b_1t}-1\right) & e^{b_1t}
\end{pmatrix} \mbox{ if } b_1\ne 0 \quad\mbox{ and } \quad e^{G_1t} =\begin{pmatrix}
1 & 0 \\
b_0t & 1
\end{pmatrix} \mbox{ if } b_1= 0.
\end{align*}
\end{theorem}

\begin{lemma}
	\label{corrr}
	If $\ell(x,dz)\equiv0$ on $\R$, then condition \eqref{lambdacond} is equivalent to $b_1 \ne -\frac{k}{2}\sigma_2$ for every $1 \le k \le 2(n-1)$. In particular, the coefficients $b_1$ and $\sigma_2$ cannot be simultaneously equal to $0$.
\end{lemma}

\subsection{Change of basis}
\label{changebasis}
The vector basis of monomials $H_n(x)$ is intuitive and allows to write down computations easily and explicitly. However, when it comes to applications, it is often more natural to choose an orthogonal basis, such as the Hermite polynomials or the Legendre polynomials, among others. Combining the properties of an orthogonal basis with the properties of polynomial processes, leads to improvements, e.g., in option pricing \cite{ackerer2020b, filipovic2020, willems2019}. Motivated by this fact, we present a result which allows to obtain the generator matrix and its exponential with respect to any basis of polynomials. This allows to employ our framework in a wider range of applications.

For $n\in \N$, we consider a set of polynomial functions $\{q_0(x), q_1(x), \dots, q_n(x)\}$ with values in $\R$ which forms a basis for $\mathrm{Pol}_n(\R)$. We then introduce the vector valued function
\begin{equation*}
	Q_n:\R \longrightarrow \R^{n+1}, \quad Q_n(x) = (q_0(x), q_1(x), \dots, q_n(x))^{\top}.
\end{equation*}
From classical linear algebra, there exists an invertible matrix $M_n\in \R^{(n+1)\times (n+1)}$ such that
\begin{equation}
	M_nH_n(x) = Q_n(x) \quad \mbox{ and } \quad M_n^{-1}Q_n(x)=H_n(x). \label{MM}
\end{equation}
We further  indicate with $J_n\in \R^{(n+1)\times(n+1)}$ the generator matrix in the sense of Theorem \ref{genm} with respect to the basis vector $Q_n(x)$, namely such that
\begin{equation}
	\G Q_n(x) = J_n Q_n(x).\label{Q}
\end{equation}
We can then prove the following result concerning $J_n$.
\begin{proposition}
\label{key}
For every $n\in \N$ and  $t\ge0$, the generator matrix $J_n$ and its matrix exponential are given by the following matrix products
\begin{equation*}
	J_n = M_nG_nM_n^{-1} \quad \mbox{and}	\quad e^{J_nt} = M_ne^{G_nt}M_n^{-1}.
\end{equation*}
\end{proposition}

\begin{example}
	Let $Q_n(x)$ be the vector basis given by Hermite polynomials. For $n=4$ we get
	\begin{equation*}
		Q_4(x) = (1,\,x,\,x^2-1,\, x^3-3x, x^4-6x^2+3)
	\end{equation*}
	while $M_4$ and $M_4^{-1}$ are respectively given by
	\begin{equation*}
		\footnotesize
		M_4 = \begin{pmatrix}
			1&0&0&0&0\\
			0&1&0&0&0\\
			-1&0&1&0&0\\
			0&-3&0&1&0\\
			3&0&-6&0&1\\
		\end{pmatrix} \quad \mbox{\normalsize and} \quad M_4^{-1} = \begin{pmatrix}
			1&0&0&0&0\\
			0&1&0&0&0\\
			1&0&1&0&0\\
			0&3&0&1&0\\
			3&0&6&0&1\\
		\end{pmatrix}.
	\end{equation*}
	We consider $\ell(x,dz)\equiv0$. By direct computation, one finds that
	\begin{equation*}
		\footnotesize
		G_4 = \begin{pmatrix}
			0&0&0&0&0\\
			b_0&b_1&0&0&0\\
			\sigma_0&2b_0+\sigma_1&2b_1+\sigma_2&0&0\\
			0&3\sigma_0&3(b_0+\sigma_1)&3(b_1+\sigma_2)&0\\
			0&0&6\sigma_0&2(2b_0+3\sigma_1)&2(2b_1+3\sigma_2)\\
		\end{pmatrix} \quad \mbox{\normalsize and}
	\end{equation*}
	\begin{equation*}
		\footnotesize
			J_4 = \begin{pmatrix}
				0&0&0&0&0\\
				b_0&b_1&0&0&0\\
				\sigma_0+2b_1+\sigma_2&2b_0+\sigma_1&2b_1+\sigma_2&0&0\\
				3\sigma_1&3(2b_1+3\sigma_2+\sigma_0)&3(b_0+\sigma_1)&3(b_1+\sigma_2)&0\\
				12\sigma_2&12\sigma_1&6(2b_1+\sigma_0+5\sigma_2)&2(2b_0+3\sigma_1)&2(2b_1+3\sigma_2)\\
			\end{pmatrix}.
	\end{equation*}
	By matrix multiplication Proposition \ref{key} can be verified.
\end{example}

\section{Applications and numerical aspects}
\label{experiments}
Having a closed and compact formula for correlators like in Theorem \ref{theoremformula} is attractive in sensitivity analysis and risk management. For example, in applications to options, the Greeks play an important role in hedging. The Greeks for options are defined as the derivatives of the price functional with respect to various parameters. In the context of path-dependent options introduced in Section \ref{pathdependent}, we shall derive in this section the expression for two of the Greeks, namely the Delta and the Theta. We shall then analyse our correlator formula from a numerical point of view, also in relation with the iterated moment formula and with a Monte Carlo approach.

\subsection{Computation of Greeks}
Two of the most common Greeks are the Delta and the Theta. The first measures the change in the option price with respect to a change in the underlying asset price. The second measures the sensitivity of the option to time to exercise.

From Section \ref{pathdependent}, the price of an Asian option can be approximated by a linear combination of conditional expectations of the form  $C^{k_0,\dots, k_m}(s_0,\dots, s_m;t) := \E\left[\left.Y(s_0)^{k_0} Y(s_1)^{k_1} \cdots Y(s_m)^{k_m}\right|\F_t\right]$, corresponding to correlators $C_{p_0,\dots, p_m}(s_0,\dots, s_m;t)$ with $p_j(x) =x^{k_j}$, $j=0,\dots, m$. Then to calculate the Delta of the Asian option, which is the partial derivative of the price functional $\Pi(t)$ with respect to the initial condition $Y(t)$, we need the derivative of $C_{p_0,\dots, p_m}(s_0,\dots, s_m;t)$ with respect to $Y(t)$.

\begin{proposition}
\label{propdelta}
For every $m\ge1$, in the same notation of Theorem \ref{theoremformula}, we have that
\begin{multline*}
	\frac{\partial C_{p_0,\dots, p_m}(s_0,\dots, s_m;t)}{\partial Y(t)}=\\
	\vec{p}_{m}^{\top}\left\{ vec^{-1} \circ e^{\tilde{G}_n^{(m)}(s_0-t)}\circ vec\left(	\frac{\partial X_n^{(m)}(Y(t))}{\partial Y(t)}\right)\right\} \prod_{k=1}^{m}e^{\tilde{G}_n^{(m-k)\top}(s_k-s_{k-1})}\left\{I_{n+1}\otimes^{m-k}\vec{p}_{m-k} \right\}
\end{multline*}
where
\begin{equation*}
\frac{\partial X_n^{(m)}(Y(t))}{\partial Y(t)} =
\left(\frac{\partial H_n(Y(t))}{\partial Y(t)}\right)^{\top} \otimes X_n^{(m-1)}(Y(t)) + H_n(Y(t))^{\top}\otimes  \frac{\partial X_n^{(m-1)}(Y(t))}{\partial Y(t)}
\end{equation*}
with $\frac{\partial H_n(Y(t))}{\partial Y(t)} = \frac{\partial X_n^{(0)}(Y(t))}{\partial Y(t)}  =  \left(0,1,\dots, n\right)\left(0, H_{n-1}(Y(t))^{\top}\right)^{\top}$.
\end{proposition}

Similarly, to compute the Theta of the Asian option, we need first to compute the derivative of $C_{p_0,\dots, p_m}(s_0,\dots, s_m;t)$ with respect to the $m+1$ time points involved, namely $s_0 <s_1<\dots < s_m$.

\begin{proposition}
	\label{proptheta}
	For every $m\ge1$, in the same notation of Theorem \ref{theoremformula}, we have that
	\begin{align*}
		& \Theta_0=\vec{p}_{m}^{\top}\left\{ vec^{-1} \circ \tilde{G}_n^{(m)}e^{\tilde{G}_n^{(m)}(s_0-t)}\circ vec\left(X_n^{(m)}(Y(t))\right)\right\} \prod_{k=1}^{m}e^{\tilde{G}_n^{(m-k)\top}(s_k-s_{k-1})}\left\{I_{n+1}\otimes^{m-k}\vec{p}_{m-k} \right\}+\\
		&\quad -\vec{p}_{m}^{\top}\left\{ vec^{-1} \circ e^{\tilde{G}_n^{(m)}(s_0-t)}\circ vec\left(	X_n^{(m)}(Y(t))\right)\right\} \tilde{G}_n^{(m-1)\top}\prod_{k=1}^{m}e^{\tilde{G}_n^{(m-k)\top}(s_k-s_{k-1})}\left\{I_{n+1}\otimes^{m-k}\vec{p}_{m-k} \right\};\\
		& \Theta_j =
		\vec{p}_{m}^{\top}\left\{ vec^{-1} \circ e^{\tilde{G}_n^{(m)}(s_0-t)}\circ vec\left(X_n^{(m)}(Y(t))\right)\right\}\cdot \\&\quad \cdot \left( \prod_{k=1}^{j-1}e^{\tilde{G}_n^{(m-k)\top}(s_k-s_{k-1})}\left\{I_{n+1}\otimes^{m-k}\vec{p}_{m-k} \right\}\tilde{G}_n^{(m-j)\top}e^{\tilde{G}_n^{(m-j)\top}(s_j-s_{j-1})} \left\{I_{n+1}\otimes^{m-j}\vec{p}_{m-j} \right\}\cdot \right.\\
		&\qquad \qquad \qquad \qquad \qquad\cdot \prod_{k=j+1}^{m}e^{\tilde{G}_n^{(m-k)\top}(s_k-s_{k-1})}\left\{I_{n+1}\otimes^{m-k}\vec{p}_{m-k} \right\}+\\
		& \quad -\prod_{k=1}^{j}e^{\tilde{G}_n^{(m-k)\top}(s_k-s_{k-1})}\left\{I_{n+1}\otimes^{m-k}\vec{p}_{m-k} \right\}\tilde{G}_n^{(m-j-1)\top}e^{\tilde{G}_n^{(m-j-1)\top}(s_{j+1}-s_{j})} \left\{I_{n+1}\otimes^{m-j-1}\vec{p}_{m-j-1} \right\}\cdot \\
		&\qquad \qquad \qquad \qquad \qquad\left.\cdot \prod_{k=j+2}^{m}e^{\tilde{G}_n^{(m-k)\top}(s_k-s_{k-1})}\left\{I_{n+1}\otimes^{m-k}\vec{p}_{m-k} \right\}\right) \quad \mbox{for } 1\le j<m ;\\
		&  \Theta_m=\vec{p}_{m}^{\top}\left\{ vec^{-1} \circ e^{\tilde{G}_n^{(m)}(s_0-t)}\circ vec\left(	X_n^{(m)}(Y(t))\right)\right\}\cdot \\ & \quad \cdot \prod_{k=1}^{m-1}e^{\tilde{G}_n^{(m-k)\top}(s_k-s_{k-1})}\left\{I_{n+1}\otimes^{m-k}\vec{p}_{m-k} \right\}\tilde{G}_n^{(0)\top}e^{\tilde{G}_n^{(0)\top}(s_m-s_{m-1})}\left\{I_{n+1}\otimes^{0}\vec{p}_{0} \right\},
	\end{align*}
where we have introduced the compact notation $\Theta_j := \frac{\partial C_{p_0,\dots, p_m}(s_0,\dots, s_m;t)}{\partial s_j}$ for $0\le j\le m.$
\end{proposition}

Then, the Delta of the Asian option is obtained by Proposition \ref{propdelta}:
\begin{equation*}
	\frac{\partial \Pi(t)}{\partial Y(t)} \approx  e^{-r(T-t)}\, \sum_{\boldsymbol{k}}\alpha_{\boldsymbol{k}}\frac{\partial C^{k_0,\dots, k_m}(s_0,\dots, s_m;t)}{\partial Y(t)}
\end{equation*}
and, similarly, the Theta is obtained by Proposition \ref{proptheta}:
\begin{equation*}
	\frac{\partial \Pi(t)}{\partial s_j} \approx  e^{-r(T-t)}\, \sum_{\boldsymbol{k}}\alpha_{\boldsymbol{k}}\frac{\partial C^{k_0,\dots, k_m}(s_0,\dots, s_m;t)}{\partial s_j} \quad \mbox{for } 0\le j \le m,
\end{equation*}
for certain coefficients $\{\alpha_{\boldsymbol{k}}\}_{\boldsymbol{k}}$ and $\boldsymbol{k} = (k_1, \cdots, k_m)$ a multi-index. A more detailed analysis for Greeks of discrete average arithmetic Asian options can be found in \cite{mio2021}, where the coefficients $\{\alpha_{\boldsymbol{k}}\}_{\boldsymbol{k}}$ are computed explicitly with respect to the basis of  Hermite polynomials.

\subsection{Numerical performances}
We analyse numerically the correlator formula in Theorem \ref{theoremformula} which explicitly gives the value for correlators, up to the computation of the exponential of the $m+1$ matrices $\tilde{G}_n^{(r)}$, $r=0, \dots, m$. This means that, assuming these matrix exponentials to be exact, the correlator formula provides \textit{the} correlator value. The same exact value can be also obtained by applying iteratively the moment formula, as we discussed in Section \ref{intropol}. We then compare from a time cost point of view these two procedures. In particular, we consider both an implementation with dense matrices and an implementation with sparse matrices. Finally, we consider a Monte Carlo approach which is compared with our correlator formula both from a time cost and an accuracy point of views.

For the experiments, we consider an Ornstein--Uhlenbeck process $Y$ defined by
\begin{equation}
	\label{OUp}
	dY(t) = (b_0+b_1Y(t))\,dt + \sqrt{\sigma_0}\,dB(t)
\end{equation}
and with model specifications
\begin{center}
	\begin{tabular}[h]{|l}
		$b_0 = +0.75$\\
		$b_1 = -5.00$\\
		$\sigma_0 = +0.01$\\
		$Y(t=0)=+0.15$\\
	\end{tabular}
\end{center}
which corresponds to a polynomial diffusion process with $\sigma_1=\sigma_2=0$ and $\ell(x, dz)\equiv0$ (see equation \eqref{poldefcond}). Moreover, for $n\ge 1$, we consider a particular case of equation \eqref{corrn} of the form
\begin{equation*}
	\label{corrn3}
	C^{n}(s_0, \dots, s_m; t) := \E\left[Y(s_0)^{n}Y(s_1)^{n}\cdot \dots \cdot Y(s_m)^{n}\left.\right|\F_t\right]
\end{equation*} 
corresponding to $p_k(x) = x^{n}=  \vec{e}_{n+1,n+1}^{\top}H_{n}(x)$, $k=0, \dots, m$. We found terms of this form in Section \ref{pathdependent} when motivating the study of correlators for pricing a financial derivative in a stochastic volatility model context. The two cases coincide for $n=1$. We also mention that Ornstein--Uhlenbeck processes are common models in finance. Among others, we find examples in modelling the electricity spot price, a non-Gaussian example being treated in \cite{benth2007}.

The Monte Carlo simulations are based on the conditional solution of equation \eqref{OUp} given by
\begin{equation*}
	\label{OUsol}
	Y(s) = Y(t) \,e^{b_1(s-t)} + \frac{b_0}{b_1}\left(e^{b_1(s-t)}-1\right)+\sqrt{\sigma_0}\int_t^se^{b_1(s-v)}dB(v), \qquad \mbox{for } s \ge t \ge 0.
\end{equation*}
We define $\Delta_k := s_k-s_{k-1}$, $k=0,\dots, m$, with $s_{-1}:=t$, and rewrite it for $s=s_k$ and $t=s_{k-1}$, namely
\begin{equation}
	\label{meanvar}
	Y(s_k) = Y(s_{k-1}) \,e^{b_1\Delta_k} + \frac{b_0}{b_1}\left(e^{b_1\Delta_k}-1\right)+\sqrt{\sigma_0}\int_{s_{k-1}}^{s_k}e^{b_1(s_k-v)}dB(v).
\end{equation}
For a fixed number of time points $m+1\ge 1$, the idea is then to sequentially simulate samples from $Y(s_k)$ according to \eqref{meanvar}, using $Y(s_{k-1})$ as starting point and the fact that, by the tower rule, it holds
\begin{equation*}
	C^{n}(s_0, \dots, s_m; t)=
	\E\left[\left.Y(s_0)^{n}\E\left[\left.Y(s_1)^{n}\E\left[\left.Y(s_2)^{n}\cdots \E\left[\left.Y(s_m)^{n}\right|\F_{s_{m-1}}\right]  \cdots  \right|\F_{s_1}\right]\right|\F_{s_0}\right] \right|\F_t\right].
\end{equation*} 
When increasing the complexity of the problem, namely, increasing $n$ and/or $m$, the Monte Carlo approach needs more simulations to gain accuracy, requiring also more time for computations. However, in order to compare different experiments, we fix the number of simulations to $N=10^4$, each of which is repeated $10^2$ times so that to get multiple values. Among these values, we select the worst one in terms of highest relative error with respect to the correlator formula, and we use it as representative of the set. For the time cost assessment, the representative is obtained by averaging the time cost required for the $10^2$ simulations. Finally, we set a tolerance to $1 \cdot 10^{-3}$ and claim that Monte Carlo \emph{fails} if the relative error is bigger than that, counting the number of failures over the $10^2$ simulations.

In Table \ref{table_time} we report the outcomes of the time-cost experiments. Here we compare the correlator formula, both with dense (Dense) and sparse (Sparse) matrices, with the iterative application of the moment formula, both with dense (Iter. dense) and sparse (Iter. sparse) matrices, and with the Monte Carlo approach (MC average). We notice that the correlator formula with dense matrices is almost comparable with the iteration with dense matrices, and the same holds with sparse matrices. However, for higher complexities, the correlator formula tends to be a bit slower. In particular, we point out that, despite what one might expect, using sparse matrices makes both the approaches slower. The main reason lies in the fact that the matrix exponential of a sparse matrix is likely not to be sparse. Hence, using a sparse matrix for what is a dense matrix instead slows down computations. However, we stress the fact that sparse matrices are crucial when increasing further the complexity of the problem. We remember indeed that, for a given $n$ and $m$, the generator matrix has dimension $(n(m+1)+1)\times(n(m+1)+1)$, so that it is not feasible to store it with a dense matrix.

We also observe from Table \ref{table_time} that the time cost for these four experiments increases when the complexity of the problem increases, as one would expect due to the dimensions of the matrices involved. However, for a fixed $m$, the time cost for the Monte Carlo approach (MC average) is almost invariant reflecting the fact that the number of time points is fixed. However, to keep the approach as general as possible, instead of computing the power $x^n$ directly we do the vector multiplication $\vec{e}_{n+1, n+1}^{\top}H_n(x)$ to reflect the situation that we would get if considering general polynomial function instead of monomials. This explains why, for a fixed $m$, when increasing the power $n$ the time cost also increases slightly.
 
In Table \ref{table_value} we report the outcomes from the accuracy tests between the correlator formula and the Monte Carlo approach. We mention indeed that the four experiments previously discussed lead to the exact same values, as one would expect since they are basically four different representations for the same exact computation. Comparing the values from the correlator formula (Formula value) with the worst over the $10^2$ outcomes of the Monte Carlo simulations (MC worst value), we notice that when increasing the complexity, the second method becomes worse and worse in terms of relative error (in parenthesis). Of course, increasing the number $N$ of simulations, one can aim at better results, but this would also mean higher time costs. In the last column (MC fails) we also notice that the number of failures in terms of the tolerance defined above increases up to almost all the calls of the method.

\begin{table}[tbp]
\begin{center}
\setlength{\tabcolsep}{3pt}
\begin{tabular}{llllllll}
\multicolumn{7}{c}{\Large \textbf{Time-cost performances}}\\[2pt]
\toprule
$\boldsymbol{m}$& $\boldsymbol{n}$&\textbf{Dense} & \textbf{Sparse} &\textbf{Iter. dense}  &\textbf{Iter. sparse} & \textbf{MC average}\\
\midrule
$ 0$ & $ 1$ & $3.372\mathrm{e}^{-04}$ & $5.254\mathrm{e}^{-03} (\approx  16\mathrm{x})$ & $3.051\mathrm{e}^{-04}(\approx   1\mathrm{x})$ & $4.904\mathrm{e}^{-03}(\approx  15\mathrm{x})$ & $1.212\mathrm{e}^{-01}(\approx 359\mathrm{x})$ \\
$ 0$ & $ 2$ & $3.451\mathrm{e}^{-04}$ & $6.700\mathrm{e}^{-03} (\approx  19\mathrm{x})$ & $3.384\mathrm{e}^{-04}(\approx   1\mathrm{x})$ & $6.199\mathrm{e}^{-03}(\approx  18\mathrm{x})$ & $1.233\mathrm{e}^{-01}(\approx 357\mathrm{x})$ \\
$ 0$ & $ 3$ & $3.612\mathrm{e}^{-04}$ & $7.879\mathrm{e}^{-03} (\approx  22\mathrm{x})$ & $3.340\mathrm{e}^{-04}(\approx   1\mathrm{x})$ & $7.459\mathrm{e}^{-03}(\approx  21\mathrm{x})$ & $1.281\mathrm{e}^{-01}(\approx 355\mathrm{x})$ \\
$ 0$ & $ 4$ & $3.819\mathrm{e}^{-04}$ & $9.242\mathrm{e}^{-03} (\approx  24\mathrm{x})$ & $3.747\mathrm{e}^{-04}(\approx   1\mathrm{x})$ & $8.932\mathrm{e}^{-03}(\approx  23\mathrm{x})$ & $1.317\mathrm{e}^{-01}(\approx 345\mathrm{x})$ \\
$ 0$ & $ 5$ & $4.053\mathrm{e}^{-04}$ & $1.072\mathrm{e}^{-02} (\approx  26\mathrm{x})$ & $3.919\mathrm{e}^{-04}(\approx   1\mathrm{x})$ & $1.028\mathrm{e}^{-02}(\approx  25\mathrm{x})$ & $1.360\mathrm{e}^{-01}(\approx 336\mathrm{x})$ \\
$ 0$ & $10$ & $4.916\mathrm{e}^{-04}$ & $1.676\mathrm{e}^{-02} (\approx  34\mathrm{x})$ & $4.883\mathrm{e}^{-04}(\approx   1\mathrm{x})$ & $1.651\mathrm{e}^{-02}(\approx  34\mathrm{x})$ & $1.570\mathrm{e}^{-01}(\approx 319\mathrm{x})$ \\
\midrule
$ 1$ & $ 1$ & $6.383\mathrm{e}^{-04}$ & $1.424\mathrm{e}^{-02} (\approx  22\mathrm{x})$ & $6.296\mathrm{e}^{-04}(\approx   1\mathrm{x})$ & $1.246\mathrm{e}^{-02}(\approx  20\mathrm{x})$ & $2.186\mathrm{e}^{-01}(\approx 342\mathrm{x})$ \\
$ 1$ & $ 2$ & $7.363\mathrm{e}^{-04}$ & $1.952\mathrm{e}^{-02} (\approx  27\mathrm{x})$ & $7.058\mathrm{e}^{-04}(\approx   1\mathrm{x})$ & $1.824\mathrm{e}^{-02}(\approx  25\mathrm{x})$ & $2.277\mathrm{e}^{-01}(\approx 309\mathrm{x})$ \\
$ 1$ & $ 3$ & $7.926\mathrm{e}^{-04}$ & $2.449\mathrm{e}^{-02} (\approx  31\mathrm{x})$ & $8.393\mathrm{e}^{-04}(\approx   1\mathrm{x})$ & $2.451\mathrm{e}^{-02}(\approx  31\mathrm{x})$ & $2.371\mathrm{e}^{-01}(\approx 299\mathrm{x})$ \\
$ 1$ & $ 4$ & $8.737\mathrm{e}^{-04}$ & $2.993\mathrm{e}^{-02} (\approx  34\mathrm{x})$ & $8.893\mathrm{e}^{-04}(\approx   1\mathrm{x})$ & $3.064\mathrm{e}^{-02}(\approx  35\mathrm{x})$ & $2.489\mathrm{e}^{-01}(\approx 285\mathrm{x})$ \\
$ 1$ & $ 5$ & $9.278\mathrm{e}^{-04}$ & $3.428\mathrm{e}^{-02} (\approx  37\mathrm{x})$ & $9.001\mathrm{e}^{-04}(\approx   1\mathrm{x})$ & $3.759\mathrm{e}^{-02}(\approx  41\mathrm{x})$ & $2.572\mathrm{e}^{-01}(\approx 277\mathrm{x})$ \\
$ 1$ & $10$ & $1.354\mathrm{e}^{-03}$ & $6.013\mathrm{e}^{-02} (\approx  44\mathrm{x})$ & $1.975\mathrm{e}^{-03}(\approx   1\mathrm{x})$ & $7.980\mathrm{e}^{-02}(\approx  59\mathrm{x})$ & $3.002\mathrm{e}^{-01}(\approx 222\mathrm{x})$ \\
\midrule
$ 2$ & $ 1$ & $1.064\mathrm{e}^{-03}$ & $2.549\mathrm{e}^{-02} (\approx  24\mathrm{x})$ & $9.527\mathrm{e}^{-04}(\approx   1\mathrm{x})$ & $2.105\mathrm{e}^{-02}(\approx  20\mathrm{x})$ & $3.213\mathrm{e}^{-01}(\approx 302\mathrm{x})$ \\
$ 2$ & $ 2$ & $1.226\mathrm{e}^{-03}$ & $3.622\mathrm{e}^{-02} (\approx  30\mathrm{x})$ & $1.116\mathrm{e}^{-03}(\approx   1\mathrm{x})$ & $3.424\mathrm{e}^{-02}(\approx  28\mathrm{x})$ & $3.319\mathrm{e}^{-01}(\approx 271\mathrm{x})$ \\
$ 2$ & $ 3$ & $1.385\mathrm{e}^{-03}$ & $4.746\mathrm{e}^{-02} (\approx  34\mathrm{x})$ & $1.122\mathrm{e}^{-03}(\approx   1\mathrm{x})$ & $4.756\mathrm{e}^{-02}(\approx  34\mathrm{x})$ & $3.488\mathrm{e}^{-01}(\approx 252\mathrm{x})$ \\
$ 2$ & $ 4$ & $1.588\mathrm{e}^{-03}$ & $5.943\mathrm{e}^{-02} (\approx  37\mathrm{x})$ & $1.377\mathrm{e}^{-03}(\approx   1\mathrm{x})$ & $6.340\mathrm{e}^{-02}(\approx  40\mathrm{x})$ & $3.658\mathrm{e}^{-01}(\approx 230\mathrm{x})$ \\
$ 2$ & $ 5$ & $1.787\mathrm{e}^{-03}$ & $7.125\mathrm{e}^{-02} (\approx  40\mathrm{x})$ & $1.493\mathrm{e}^{-03}(\approx   1\mathrm{x})$ & $7.975\mathrm{e}^{-02}(\approx  45\mathrm{x})$ & $3.727\mathrm{e}^{-01}(\approx 209\mathrm{x})$ \\
$ 2$ & $10$ & $7.786\mathrm{e}^{-03}$ & $1.283\mathrm{e}^{-01} (\approx  16\mathrm{x})$ & $3.979\mathrm{e}^{-03}(\approx   1\mathrm{x})$ & $1.870\mathrm{e}^{-01}(\approx  24\mathrm{x})$ & $4.380\mathrm{e}^{-01}(\approx  56\mathrm{x})$ \\
\midrule
$ 3$ & $ 1$ & $1.611\mathrm{e}^{-03}$ & $3.930\mathrm{e}^{-02} (\approx  24\mathrm{x})$ & $1.238\mathrm{e}^{-03}(\approx   1\mathrm{x})$ & $3.261\mathrm{e}^{-02}(\approx  20\mathrm{x})$ & $4.216\mathrm{e}^{-01}(\approx 262\mathrm{x})$ \\
\midrule
$ 4$ & $ 1$ & $2.101\mathrm{e}^{-03}$ & $5.545\mathrm{e}^{-02} (\approx  26\mathrm{x})$ & $1.474\mathrm{e}^{-03}(\approx   1\mathrm{x})$ & $4.468\mathrm{e}^{-02}(\approx  21\mathrm{x})$ & $5.259\mathrm{e}^{-01}(\approx 250\mathrm{x})$ \\
\midrule
$ 5$ & $ 1$ & $2.910\mathrm{e}^{-03}$ & $7.554\mathrm{e}^{-02} (\approx  26\mathrm{x})$ & $1.887\mathrm{e}^{-03}(\approx   1\mathrm{x})$ & $6.005\mathrm{e}^{-02}(\approx  21\mathrm{x})$ & $6.175\mathrm{e}^{-01}(\approx 212\mathrm{x})$ \\
\midrule
$10$ & $ 1$ & $6.423\mathrm{e}^{-02}$ & $2.302\mathrm{e}^{-01} (\approx   4\mathrm{x})$ & $4.005\mathrm{e}^{-03}(\approx   0\mathrm{x})$ & $1.634\mathrm{e}^{-01}(\approx   3\mathrm{x})$ & $1.112\mathrm{e}^{+00}(\approx  17\mathrm{x})$ \\
\bottomrule
\end{tabular}
\end{center}
\caption{Time-cost performances for the correlator formula with dense (\textbf{Dense}) and sparse (\textbf{Sparse}) matrices, for the iterative application of the moment formula with dense (\textbf{Iter. dense}) and sparse (\textbf{Iter. sparse}) matrices, and for the the Monte Carlo approach (\textbf{MC average}). In parenthesis, how many times the different approaches are slower than the correlator formula with dense matrices.
\label{table_time}
}
\end{table}

\begin{table}[tb]
\begin{center}
\setlength{\tabcolsep}{3pt}
\begin{tabular}{llccr}
\multicolumn{5}{c}{\Large \textbf{Accuracy performances}}\\[2pt]
\toprule
$\boldsymbol{m}$& $\boldsymbol{n}$&\textbf{Formula value}  &\textbf{MC worst value}  &  \textbf{MC fails}\\
\midrule
$ 0$ & $ 1$ & $1.500\mathrm{e}^{-01}$ & $1.490\mathrm{e}^{-01} \; (0.66\%)$ & $66/100$\\
$ 0$ & $ 2$ & $2.350\mathrm{e}^{-02}$ & $2.378\mathrm{e}^{-02}  \; ( 1.18\%)$ &  $77/100$\\
$ 0$ & $ 3$ & $3.825\mathrm{e}^{-03}$  & $3.751\mathrm{e}^{-03} \; (1.93\%)$ &  $91/100$\\
$ 0$ & $ 4$ & $6.442\mathrm{e}^{-04}$  & $6.567\mathrm{e}^{-04}  \; ( 1.94\%)$ &  $93/100$\\
$ 0$ & $ 5$ & $1.119\mathrm{e}^{-04}$  & $1.092\mathrm{e}^{-04} \; ( 2.43\%)$ &  $92/100$\\
$ 0$ & $10$ & $2.618\mathrm{e}^{-08}$  & $2.828\mathrm{e}^{-08}  \; ( 8.02\%)$ &  $98/100$\\
\midrule
$ 1$ & $ 1$ & $2.258\mathrm{e}^{-02}$  & $2.278\mathrm{e}^{-02}  \; ( 0.88\%)$ &  $79/100$\\
$ 1$ & $ 2$ & $5.594\mathrm{e}^{-04}$  & $5.505\mathrm{e}^{-04} \; ( 1.60\%)$ &  $94/100$\\
$ 1$ & $ 3$ & $1.503\mathrm{e}^{-05}$   & $1.460\mathrm{e}^{-05} \; ( 2.82\%)$ &  $91/100$\\
$ 1$ & $ 4$ & $4.338\mathrm{e}^{-07}$   & $4.164\mathrm{e}^{-07} \; ( 4.00\%)$ &  $95/100$\\
$ 1$ & $ 5$ & $1.337\mathrm{e}^{-08}$   & $1.268\mathrm{e}^{-08}  \; ( 5.09\%)$ &  $97/100$\\
$ 1$ & $10$ & $8.366\mathrm{e}^{-16}$  & $6.916\mathrm{e}^{-16} \; ( 17.3\%)$ &  $99/100$\\
\midrule
$ 2$ & $ 1$ & $3.436\mathrm{e}^{-03}$   & $3.467\mathrm{e}^{-03} \; ( 0.92\%)$ &  $81/100$\\
$ 2$ & $ 2$ & $1.383\mathrm{e}^{-05}$  & $1.346\mathrm{e}^{-05}\; ( 2.62\%)$ &  $92/100$\\
$ 2$ & $ 3$ & $6.372\mathrm{e}^{-08}$  & $6.662\mathrm{e}^{-08} \; ( 4.55\%)$ &  $95/100$\\
$ 2$ & $ 4$ & $3.310\mathrm{e}^{-10}$ & $3.529\mathrm{e}^{-10} \; ( 6.60\%)$ &  $98/100$\\
$ 2$ & $ 5$ & $1.914\mathrm{e}^{-12}$ & $1.742\mathrm{e}^{-12} \; ( 9.00\%)$ &  $99/100$\\
$ 2$ & $10$ & $4.638\mathrm{e}^{-23}$   & $8.030\mathrm{e}^{-23} \; ( 73.2\%)$ &  $100/100$\\
\midrule
$ 3$ & $ 1$ & $5.291\mathrm{e}^{-04}$  & $5.352\mathrm{e}^{-04} \; ( 1.15\%)$ &  $86/100$\\
\midrule
$ 4$ & $ 1$ & $8.245\mathrm{e}^{-05}$ & $8.097\mathrm{e}^{-05} \; ( 1.80\%)$ &  $87/100$\\
\midrule
$ 5$ & $ 1$ & $1.300\mathrm{e}^{-05}$   & $1.326\mathrm{e}^{-05} \; ( 2.02\%)$ &  $88/100$\\
\midrule
$10$ & $ 1$ & $1.477\mathrm{e}^{-09}$  & $1.413\mathrm{e}^{-09}\; ( 4.36\%)$ &  $91/100$\\
\bottomrule
\end{tabular}
\end{center}
\caption{Accuracy performances for the correlator formula (\textbf{Formula value}) compared with a Monte Carlo approach (\textbf{MC worst value}). In parenthesis, the relative error with respect to the correlator formula. We then report the ratio of values with relative error exceeding the tolerance (\textbf{MC fails}). 
\label{table_value}
}
\end{table}

\begin{remark}
	As observed above, the correlator $C^{n}(s_0, \dots, s_m; t)$ coincides for $n=1$ with the terms encountered in Section~\ref{pathdependent} in the setting of a stochastic volatility model. We then chose the parameters for the model in equation \eqref{OUp} so that to make it relevant in view of that application. However, the choice of a Gaussian Ornstein--Uhlenbeck process is for simplicity and illustration, and not intended as a precise stochastic volatility model. We also point out that in equation \eqref{intcorr2} one has in practise to truncate the summation to a certain $\bar{k}\in \N$ to get an approximated option price. From Table \ref{table_value}, we notice that for $n=1$ the correlator values are very small in this configuration. This means that $\bar{k}$ can be chosen small, possibly not bigger than 5. For small values of $k$ (that means, small values of $m$), the iterated integral of the correlator formula (see the terms in equation \eqref{intcorr2}) can be computed by hand, avoiding additional error due to some multidimensional integration technique necessary otherwise.
\end{remark}

\section{Conclusions}
\label{conc}
We have derived an explicit formula for computing correlators of polynomial processes consisting of linear combinations of exponentials of the generator matrix associated with the polynomial process. Our analysis is based on a recursive use of the moment formula for conditional expectations of polynomial processes along with the introduction of two new linear operators, respectively, the L-eliminating and the L-duplicating matrices. The closed-form expression of the correlators of polynomial processes is attractive in studies of options and risk management. The connection to Hankel matrices open our studies of correlators of polynomial processes to other areas as well. 

We want to stress that a closed formula allows to make analysis with respect to the variables and parameters involved. The correlator formula in Theorem \ref{theoremformula} depends on the polynomial jump-diffusion coefficients, $b$, $\sigma$ and $\xi$, as introduced in equation \eqref{poldefcond}. But it also depends on the time points $t < s_0 < s_1 < \cdots < s_m<T<+\infty$. This fact may be exploited for option price analysis as we demonstrate. Indeed, the correlator formula allows for explicit computation of derivatives, and lend themselves to the calculation of option Greeks.

\paragraph*{Acknowledgements} Christa Cuchiero is thanked for interesting discussions. We are also grateful for the careful reading and the suggestions of three anonymous referees, which have led to a significant improvement on the presentation of the paper.

\appendix
\section{Some combinatorial properties}
\label{appApol}
We report in this section some combinatorial properties concerning the L-eliminating and the L-duplicating matrices and the matrix of functions $X_n^{(r)}(x)$. However, these are not necessary to understand the main part of the article. 

\subsection{The L-eliminating  and L-duplicating matrices}
The total number of elements of $E_{n,m}$ is equal to $nm(n+m-1)$. When multiplying $E_{n,m}$ with a matrix $A\in \R^{n\times m}$, we select the elements of $A$ which are in the first column and last row, that account for exactly $n+m-1$ terms. That means that $E_{n,m}$ must have exactly $n+m-1$ elements equal to $1$ and the rest must be zeros. Moreover, for $A\in \mathcal{A}_{n,m}$, the matrix $D_{n,m}$ duplicates each element $a_k$ in $vecL(A)$ as many times as the number of elements in the $k$-th skew-diagonal of $A$, $k=1, \dots, n+m-1$. Then in the $k$-th column of $D_{n,m}$ there are as many $1$'s as the number of elements in the $k$-th skew-diagonal of $A$, while the remaining elements are all $0$'s. We shall be more precise in the following lemmas.

\begin{lemma}
	\label{numlem}
	For $E_{n,m}\in \R^{(n+m-1)\times nm}$, the number of non-zero elements is $n+m-1$.
\end{lemma}

\begin{lemma}
	\label{card}
	If $n\ne m$, then the number of $1$'s in the $k$-th column of $D_{n,m}$ is
	\begin{equation}
		\label{cases1}
		\begin{cases}
			k & \mbox{ for } 1\le k\le \min(n,m)-1\\
			\min(n,m)&  \mbox{ for } \min(n,m)\le k\le \max(n,m)-1\\
			n+m-k&  \mbox{ for } \max(n,m)\le k\le n+m-1
		\end{cases}.
	\end{equation}
	If $n=m$, then the following formula holds instead:
	\begin{equation}
		\label{cases2}
		\begin{cases}
			k & \mbox{ for } 1\le k\le n-1\\
			2n-k&  \mbox{ for } n\le k\le 2n-1
		\end{cases}.
	\end{equation}
	In particular, if $n= m$, then the number of $1$'s in the $k$-th column corresponds to the coefficient of the $(k-1)$-th power of $x$ in the power expansion $\left(\sum_{\alpha=0}^{n-1} x^{\alpha}\right)^2$.
\end{lemma}
\begin{proof}
	Let $n\ne m$ and $A\in \mathcal{A}_{n,m}$. In the $k$-th column of $D_{n,m}$ there are as many $1$'s as the number of elements in the $k$-th skew-diagonal of $A$. Denoting with $a_k$ the value of the elements on the $k$-th skew-diagonal of $A$, $k=1,\dots, n+m-1$, equation \eqref{cases1} gives the cardinality of each $a_k$. Then the sum of the elements in equation \eqref{cases1} should give the total number of elements in $A$, that is $nm$:
	\begin{align*}
		\label{sums}
		&\sum_{k=1}^{\min(n,m)-1} k +\sum_{k=\min(n,m)}^{\max(n,m)-1}\min(n,m)+\sum_{k=\max(n,m)}^{n+m-1}(m+n-k)\\
		& = \frac{\min(n,m)\left(\min(n,m)-1\right)}{2} + \min(n,m)\left(\max(n,m)-\min(n,m)\right) + \frac{\min(n,m)\left(\min(n,m)+1\right)}{2}\\
		& = \max(n,m)\min(n,m)=nm,
	\end{align*}
	where for the third sum we used the change of variables $k':= m+n-k$ so that $\sum_{k=\max(n,m)}^{n+m-1}(m+n-k)=\sum_{k'=1}^{n+m-\max(m,n)}k'$, and the fact that
	$m+n-\max(n,m)=\min(n,m)$. The case $n=m$ is similar, so that the first part of the lemma is proved.
	
	We need now to prove that the numbers in equation \eqref{cases2} correspond to the coefficients in the power expansion $\left(\sum_{\alpha=0}^{n-1} x^{\alpha}\right)^2$. We proceed by induction on the matrix dimension $n\ge 2$ ($n=1$ is trivial).
	\begin{itemize}[leftmargin=*]
		\item $n=2$: a matrix $A \in \mathcal{A}_{2,2}$ is of the form $A=\begin{pmatrix}
			a_1 & a_2\\
			a_2 & a_3
		\end{pmatrix}$ and the cardinality of the $a_k$, $k=1,2,3$, is $1-2-1$, which correspond to the coefficients of the polynomial $\left(\sum_{\alpha=0}^{1} x^{\alpha}\right)^2= (1+x)^2 = 1+2x+x^2$.
		\item $n\rightarrow n+1$: we indicate with $A_n$ a general matrix in $\mathcal{A}_{n}$, and with $A_{n+1}$ a general matrix in $\mathcal{A}_{n+1}$. Then $A_n$ and $A_{n+1}$ can be represented as
		\begin{equation*}
			\NiceMatrixOptions{transparent}
			A_{n}=
			\begin{pmatrix}
				a_1 & a_2 &  \cdots & a_n\\
				a_2 & a_3 & \cdots & a_{n+1}\\
				\vdots &\vdots & \iddots & \vdots \\
				a_n & a_{n+1} &  \cdots & a_{2n-1}
			\end{pmatrix}\;
			\mbox{ and }\;
			A_{n+1}=
			\left(\begin{NiceArray}{cccc|c}%
				a_1     &    a_2     &  \cdots     & a_n      &     a_{n+1}\\
				a_2      &    a_3     & \cdots     & a_{n+1} &    \vdots\\
				\vdots &   \vdots  & \iddots    & \vdots   & a_{2n-1} \\
				a_n      & a_{n+1}  &  \cdots    & a_{2n-1} &   a_{2n}  \\
				\hline
				a_{n+1}&\cdots    &  a_{2n-1}    & a_{2n}     & a_{2n+1} 
			\end{NiceArray}\right).
		\end{equation*}
		In particular, $A_n$ contains the entries $a_k$ from $k=1$ to $k=2n-1$, whose cardinality, by induction hypothesis, corresponds to the coefficients of the polynomial $\left(\sum_{\alpha=0}^{n-1} x^{\alpha}\right)^2$. Moreover, the entries $a_k$ from $k=n+1$ to $k=2n-1$ appear two extra times in $A_{n+1}$: once in the last row and once in the last column. Finally, in $A_{n+1}$ we have two additional entries, $a_{2n}$ and $a_{2n+1}$, that are not in $A_n$, and whose cardinality is, respectively, $2$ and $1$. To summarize, the cardinality of the entry $a_k$ in $A_{n+1}$, $k=1,\dots, 2n+1$, corresponds to the $(k-1)$-th power of $x$ in the following polynomial:
		\begin{equation*}
			\left(\sum_{\alpha=0}^{n-1} x^{\alpha}\right)^2 +2 \left(x^n+\dots+x^{2n-2}\right) + 2x^{2n-1}+x^{2n} 
			 =  \left(\sum_{\alpha=0}^{n-1} x^{\alpha}\right)^2 +2 x^n\left(\sum_{\alpha=0}^{n-1} x^{\alpha}\right) +x^{2n} = \left(\sum_{\alpha=0}^{n} x^{\alpha}\right)^2,
		\end{equation*} 
		which concludes the proof.
	\end{itemize} 	
\end{proof}

\begin{example}
	\label{E2ex}
	Let $n=m=3$. After some technical calculations, we get
	\begin{equation*}
		\label{E2}
		\footnotesize
		E_{3,3}=
		\begin{pmatrix}
			1& 0 & 0 & 0& 0 & 0 & 0& 0 & 0\\
			0&1& 0 & 0 & 0 & 0 & 0& 0 & 0\\
			0& 0 & 1 & 0 & 0 & 0& 0 & 0 & 0\\
			0& 0 & 0 & 0 & 0 & 1& 0 & 0 & 0\\
			0& 0 & 0 & 0 & 0 & 0& 0 & 0 & 1
		\end{pmatrix} \qquad \mbox{\normalsize  and} \qquad D_{3,3}  =
	\begin{pmatrix}
	1& 0 & 0 & 0& 0\\
	0& 1 & 0& 0 & 0\\
	0& 0 & 1& 0 & 0\\
	0& 1 & 0& 0 & 0\\
	0& 0 & 1& 0 & 0\\
	0& 0 & 0& 1 & 0\\
	0& 0 & 1& 0 & 0\\
	0& 0 & 0& 1 & 0\\
	0& 0 & 0& 0 & 1
\end{pmatrix}.
	\end{equation*}
	The number of non-zero elements in $E_{3,3}$ is $n+m-1=5$ a stated in Lemma \ref{numlem}, and the number of $1$'s in each column of $D_{3,3}$ is $1-2-3-2-1$, like the coefficients in the polynomial expansion $$\left(\sum_{\alpha=0}^2x^{\alpha}\right)^2 = \left(1+x+x^2\right)^2= 1+2x+3x^2+2x^3+x^4,$$ according to Lemma \ref{card}.
\end{example}

\subsection{The matrix of functions}
From Proposition \ref{shapeX}, $X_n^{(r)}(x)$ is a rectangular block matrix composed by $(n+1)^{r-1}$ blocks $B_{n,r}^{(k)}(x) \in \mathcal{A}_{n+1,n+1}$, which can be expressed by $B_{n,r}^{(k)}(x)= x^{j_k}X_n^{(1)}(x)$, for a certain $j_k \in \{0, \dots, (r-1)n\}$. We shall now give the cardinality of each block and an explicit formula for $j_k$. We denote by $\bmod$ and $\%$ the operators which, respectively, return the remainder and the quotient of the division between two natural numbers, namely for $a,b,c,d\in \N$, $c=a \bmod b$ and $d=a\%b$ means that $a = bd+c$.
\begin{lemma}
	\label{shapeX2}
	For every $n, r\ge 1$, each block of the form $x^j X_n^{(1)}(x)$, for $j=0, \dots, (r-1)n$, is repeated with cardinality $\beta_{n,r}^{(j)}:= \# \{k: j_k =j\}$, that is equal to the coefficient of the $j$-th power of $x$ in the polynomial expansion $\left(\sum_{\alpha=0}^n x^{\alpha}\right)^{r-1}.$
\end{lemma}
\begin{proof}
	We proceed by induction on $r\ge 1$.
	\begin{itemize}[leftmargin=*]
		\item $r=1$: $X_n^{(1)}(x)= x^0X_n^{(1)}(x)$ is composed by only one block,
		and the polynomial $\left(\sum_{\alpha=0}^n x^{\alpha}\right)^0 = 1$ has the only coefficient $1$, that is the cardinality of the unique block.
		\item $r \to  r+1$: we assume the statement holds for $r$, then we need to prove that for each $j=0, \dots, rn$, the cardinality of the block of the form $x^j X_n^{(1)}(x)$ corresponds to the coefficient of the $j$-th power of $x$ in the polynomial expansion $\left(\sum_{\alpha=0}^n x^{\alpha}\right)^{r}$. Since the Kronecker product is associative, we write
		\begin{align}
			\notag
			X_n^{(r+1)}(x) &= H_n(x)^{\top} \otimes^{r+1}H_n(x) = H_n(x)^{\top}\otimes \left(H_n(x)^{\top} \otimes^{r}H_n(x)\right) = H_n(x)^{\top}\otimes X_n^{(r)}(x)\\
		&= \left(
			\begin{NiceArray}{c:c:c:c:c}
				X_n^{(r)}(x) & xX_n^{(r)}(x) & x^2X_n^{(r)}(x)& \cdots\cdots & x^nX_n^{(r)}(x)
			\end{NiceArray}\right).\label{Hr}
		\end{align}
		We know by induction hypothesis that each block of $X_n^{(r)}(x)$ of the form $x^{j} X_n^{(1)}(x)$, $j \in \{0, \dots, (r-1)n\}$, has cardinality according to the $j$-th coefficient of the polynomial expansion $\left(\sum_{\alpha=0}^n x^{\alpha}\right)^{r-1}$. It is clear that such cardinality shifts to the upper power when we multiply $X_n^{(r)}(x)$ by $x$, and it shifts by two positions when we multiply $X_n^{(r)}(x)$ by $x^2$, and so on (for example, if the block $x^{j} X_n^{(1)}(x)$ has cardinality $\beta_{n,r}^{(j)}$ in $X_n^{(r)}(x)$, then $\beta_{n,r}^{(j)}$ is the cardinality of $x^{j+1} X_n^{(1)}(x)$ in $xX_n^{(r)}(x)$, of $x^{j+2} X_n^{(1)}(x)$ in $x^2X_n^{(r)}(x)$, and so on). To summarize, we say that in $X_n^{(r+1)}(x)$ the cardinality for the block of the form $x^jX_n^{(1)}(x)$ corresponds to the coefficient of the $j$-th power in the following polynomial:
		\begin{equation*}
			\left(\sum_{\alpha=0}^n x^{\alpha}\right)^{r-1} + x\left(\sum_{\alpha=0}^n x^{\alpha}\right)^{r-1} + \dots + x^n \left(\sum_{\alpha=0}^n x^{\alpha}\right)^{r-1}= \left(\sum_{\alpha=0}^n x^{\alpha}\right)^{r},
		\end{equation*}
		which concludes the proof.
	\end{itemize} 
\end{proof}

\begin{lemma}
	\label{expjk}
	For every $n,r\ge 1$, the blocks $B_{n,r}^{(k)}(x)$ composing  the matrix of functions $X_n^{(r)}(x)$ are of the form $B_{n,r}^{(k)}(x)= x^{\gamma_{n,r}^{(k)}}X_n^{(1)}(x)$, for $\gamma_{n,r}^{(k)} \in \{0, \dots, (r-1)n\}$ given by the formula
	\begin{equation*}
		\label{power}
		\gamma_{n,r}^{(k)} = \sum_{j=0}^{r-1}  \left\{(k-1) \bmod(n+1)^{r-j} \right\}\% (n+1)^{r-1-j} \quad \mbox{for }k=1,\dots (n+1)^{r-1}.
	\end{equation*}
\end{lemma}
\begin{proof}
	By associativity property of the Kronecker product, we write $X_n^{(r)}(x) = \left(H_n(x)^{\top}\right)^{\otimes (r-1)} \otimes X_n^{(1)}(x)$, where $\left(H_n(x)^{\top}\right)^{\otimes (r-1)}$ is a row vector in $\R^{(n+1)^{r-1}}$ with elements the monomials $x^{\gamma_{n,r}^{(k)}}$ whose exponents we want to study, $k=1,\dots (n+1)^{r-1}$. We focus on $\left(H_n(x)^{\top}\right)^{\otimes r}$ for simplicity: we need then to prove that the $k$-th element of $\left(H_n(x)^{\top}\right)^{\otimes r}$ is a monomial with exponent given by
	\begin{equation}
		\label{power2}
		P_{n,r}^{(k)} = \sum_{j=0}^{r}  \left\{(k-1) \bmod(n+1)^{r-j+1} \right\}\% (n+1)^{r-j} \quad \mbox{for } k=1,\dots (n+1)^{r}.
	\end{equation}
	The result will follow noticing that $\gamma_{n,r}^{(k)}  = P_{n,r-1}^{(k)}$. We proceed by induction on $r\ge1$.
	\begin{itemize}[leftmargin=*]
		\item $r=1$: for $H_n(x)$ we easily notice that the exponent of the $k$-th term equals $k-1$, $k=1,\dots,(n+1)$. We now look at equation \eqref{power2}:
		\begin{align*}
			P_{n,1}^{(k)}& = \sum_{j=0}^{1}  \left\{(k-1) \bmod(n+1)^{2-j} \right\}\% (n+1)^{1-j}\notag\\
			& = \left\{(k-1) \bmod(n+1)^{2} \right\}\% (n+1) + \left\{(k-1) \bmod(n+1)\right\}\% (n+1)^0
		\end{align*}
		where the first term is $0$ and the second one is $k-1$, since $0\le k-1\le n$.
		
		\item $r\to r+1$: we now assume formula \eqref{power2} holds for $r$. By associativity property, $\left(H_n(x)^{\top}\right)^{\otimes (r+1)} =H_n(x)^{\top} \otimes \left(H_n(x)^{\top}\right)^{\otimes r}$, so that $\left(H_n(x)^{\top}\right)^{\otimes (r+1)}$ is the block vector
		\begin{equation*}
			\left(H_n(x)^{\top}\right)^{\otimes (r+1)} = \left(
			\begin{NiceArray}{c:c:c:c:c}
				\left(H_n(x)^{\top}\right)^{\otimes r} & x\left(H_n(x)^{\top}\right)^{\otimes r}& x^2\left(H_n(x)^{\top}\right)^{\otimes r}& \cdots\cdots & x^n\left(H_n(x)^{\top}\right)^{\otimes r}
			\end{NiceArray}\right).
		\end{equation*}
		For each element in $\left(H_n(x)^{\top}\right)^{\otimes (r+1)}$, the exponent is $P_{n,r}^{(k)} + \alpha$ with $\alpha\in\{0,\dots,n\}$. However, we notice that $\left(H_n(x)^{\top}\right)^{\otimes r}$ has index $k = 1,\dots,(n+1)^r$, while $\left(H_n(x)^{\top}\right)^{\otimes (r+1)}$ has index $\hat{k} = 1,\dots,(n+1)^{r+1}$.
		Then in formula \eqref{power2} we must substitute $(k-1) = (\hat{k}-1)\bmod (n+1)^r$. Moreover, $\alpha = (\hat{k}-1)\% (n+1)^r$. Putting all these considerations together, we write that (we omit the superscript $\hat{}$ on the index $k$):
		\begin{equation}
			\label{power3}
			P_{n,r+1}^{(k)} = (k-1)\% (n+1)^r \\+ \sum_{j=0}^{r}  \left\{(k-1) \bmod (n+1)^r\bmod(n+1)^{r-j+1} \right\}\% (n+1)^{r-j}.
		\end{equation}
		In particular, it is easy to see that
		\begin{equation*}
			(k-1) \bmod (n+1)^r\bmod(n+1)^{r-j+1}=
			\begin{cases}
				(k-1) \bmod(n+1)^{r-j+1} & j\ge 1\\
				(k-1) \bmod (n+1)^r & j=0
			\end{cases},
		\end{equation*}
		so that equation \eqref{power3} becomes
		\begin{multline}
			\label{power4}
			P_{n,r+1}^{(k)} = (k-1)\% (n+1)^r + (k-1) \bmod (n+1)^r\% (n+1)^r +\\+  \sum_{j=1}^{r}  \left\{(k-1) \bmod(n+1)^{r-j+1} \right\}\% (n+1)^{r-j}.
		\end{multline}
		Moreover, since $1\le k\le(n+1)^{r+1}$, we also notice that 
		\begin{align*}
			& (k-1)\% (n+1)^r = (k-1) \bmod(n+1)^{r+1}\% (n+1)^r,\\
			& (k-1) \bmod (n+1)^r\% (n+1)^r (=0)= (k-1) \bmod (n+1)^{r+2}\% (n+1)^{r+1},
		\end{align*}
		and equation \eqref{power4} can be rewritten as
		\begin{align*}
			P_{n,r+1}^{(k)} &= \sum_{j=-1}^{r}  \left\{(k-1) \bmod(n+1)^{r-j+1} \right\}\% (n+1)^{r-j}\\&=\sum_{j=0}^{r+1}  \left\{(k-1) \bmod(n+1)^{r-j+2} \right\}\% (n+1)^{r-j+1},
		\end{align*}
		which concludes the proof.
	\end{itemize}
\end{proof}

\section{Proofs}
\label{el_dup}
We report in this section the proofs of the main results of the paper together with some additional results that are needed for the proofs.
\begin{proposition}
	For the matrices $A, B \in \R^{n\times m}$ with elements, respectively, $[A]_{i,j} = a_{i,j}$ and $[B]_{i,j} = b_{i,j}$, $1\le i \le n$ and $1 \le j \le m$, and $\vec{x}, \vec{y}$ vectors of any order, we have the following properties:
	\begin{subequations}
		\begin{align}
			&A = \sum_{i=1}^{n}\sum_{j=1}^m a_{i,j}\vec{e}_{n,i}\vec{e}_{m,j}^{\top}\label{pr1}\\
			&(vec(A))^{\top}vec(B) = tr(A^{\top}B)\label{pr4}\\
			&\vec{x} \otimes \vec{y} = vec(\vec{y}\,\vec{x}^{\top})\label{pr2}\\
			&\vec{x}\otimes \vec{y}^{\top}= \vec{x}\,\vec{y}^{\top} = \vec{y}^{\top} \otimes \vec{x}\label{pr3}
		\end{align}
	\end{subequations}
	where $tr$ denotes the trace operator. Moreover, for every $A \in \R^{p\times q}$, $B \in \R^{r\times s}$, $C \in \R^{q\times k}$ and $D \in \R^{s\times l}$, the mixed-product property holds:
	\begin{equation}
		\label{mix}
		\left(A \otimes B\right)\left(C \otimes D\right) = \left(AC\right) \otimes \left(BD\right).
	\end{equation}
\end{proposition}
\begin{proof}\renewcommand{\qedsymbol}{}
	We refer to \cite[Section 2]{magnus1980} and \cite[Lemma 4.2.10]{horn1991}.
\end{proof}

\subsubsection*{Proof of Theorem \ref{Enmth}}
\begin{proof}
	We will prove the existence of the matrix $E_{n,m}$ by proving its explicit definition. We first give a characterization of $vecL(A)$ in terms of the unitary vectors. By equation \eqref{pr1}, one easily see that
	\begin{equation}
		\label{vecLp}
		vecL(A) = \sum_{i=1}^n a_{i,1}\vec{e}_{n+m-1,i}+\sum_{i=2}^m a_{n,i}\vec{e}_{n+m-1,n+i-1}.
	\end{equation} 
	In particular, $a_{i,1}= \vec{e}_{n,i}^{\top} A \vec{e}_{m,1}= tr(\vec{e}_{m,1}\vec{e}_{n,i}^{\top} A)$ and $a_{n,i}= \vec{e}_{n,n}^{\top} A \vec{e}_{m,i}= tr(\vec{e}_{m,i}\vec{e}_{n,n}^{\top} A)$. 
	Moreover, by property \eqref{pr4} we write that 
	\begin{align*}
		&tr(\vec{e}_{m,1}\vec{e}_{n,i}^{\top} A) = tr((\vec{e}_{n,i}\vec{e}_{m,1}^{\top})^{\top} A) = vec(\vec{e}_{n,i}\vec{e}_{m,1}^{\top})^{\top}vec(A),\\
		&tr(\vec{e}_{m,i}\vec{e}_{n,n}^{\top} A) = tr((\vec{e}_{n,n}\vec{e}_{m,i}^{\top})^{\top} A) = vec(\vec{e}_{n,n}\vec{e}_{m,i}^{\top})^{\top}vec(A).
	\end{align*}
	By combining these results with  equations \eqref{pr2}, \eqref{pr3} and \eqref{vecLp}, we get that
	\begin{align*}
		\label{vecLp2}
		vecL(A) &= \left(\sum_{i=1}^n \vec{e}_{n+m-1,i}vec(\vec{e}_{n,i}\vec{e}_{m,1}^{\top})^{\top} +\sum_{i=2}^m \vec{e}_{n+m-1,n+i-1}vec(\vec{e}_{n,n}\vec{e}_{m,i}^{\top})^{\top}\right)vec(A)\\
		&= \left(\sum_{i=1}^n \vec{e}_{n+m-1,i}\otimes \vec{e}_{m,1}^{\top}\otimes \vec{e}_{n,i}^{\top}+\sum_{i=2}^m \vec{e}_{n+m-1,n+i-1}\otimes\vec{e}_{m,i}^{\top}\otimes \vec{e}_{n,n}^{\top}\right)vec(A).
	\end{align*} 
	Then the L-eliminating matrix $E_{n,m}$ satisfying the implicit definition in equation \eqref{Enm2} is the one in equation \eqref{Enmex2}. This concludes the proof.
\end{proof}

\subsubsection*{Proof of Theorem \ref{Dnth}}
\begin{proof}
	We construct the matrix $D_{n,m}$ explicitly. Since $A\in \mathcal{A}_{n,m}$, the elements of $A$ along the skew-diagonals coincide and $A$ has exactly $n+m-1$ skew-diagonals, leading to at most $n+m-1$ different values. In the notation of Definition \ref{spaceA}, let $a_k, \, k=1, \dots, (n+m-1)$, such that $vecL(A) = (a_1, a_2, \dots, a_{n+m-1})^{\top}\in\R^{n+m-1}$. For $1\le i \le n$ and $1\le j \le m$, it holds that 
	\begin{equation}
		[A]_{i,j}=a_{i+j-1} = [vecL(A)]_{i+j-1}=vecL(A)^{\top}\vec{e}_{n+m-1,i+j-1}= \vec{e}_{n+m-1,i+j-1}^{\top}vecL(A).\label{vij}
	\end{equation}
	We notice that $\vec{e}_{m,j}\otimes \vec{e}_{n,i}$ is the unitary vector in $\R^{nm}$ with $1$ in position $n(j-1)+i$ and $0$ elsewhere. We use this fact together with equation \eqref{vij} to express the vectorization of $A$ as follows:
	\begin{align*}
		vec(A)= \sum_{i=1}^{n}\sum_{j=1}^{m} a_{i+j-1}\vec{e}_{m,j}\otimes \vec{e}_{n,i}= \left(\sum_{i=1}^{n}\sum_{j=1}^{m} \left(\vec{e}_{m,j}\otimes \vec{e}_{n,i}\right)\vec{e}_{n+m-1,i+j-1}^{\top}\right)vecL(A).
	\end{align*}
	By equation \eqref{pr3}, we define $D_{n,m}$ as in equation \eqref{Dnex2}, which proves the theorem.
\end{proof}

\subsubsection*{Proof of Proposition \ref{EDp}}
\begin{proof}
	For the second part of the statement, the proof is straightforward from the definitions of $E_{n,m}$ and $D_{n,m}$, namely combining equations \eqref{Enm2} and \eqref{Dn2}. We resume the situation as follows:
	\begin{equation*}
		vec(A)\xrightarrow{\;E_{n,m}\;} vecL(A) \xrightarrow{\;D_{n,m}\;} vec(A)
	\end{equation*} 
	so that $D_{n,m} E_{n,m}\in\R^{nm\times nm}$ acts like an identity operator on $vec(A)$ for each $A \in \mathcal{A}_{n,m}$.
	
	We prove now the first part of the statement. By equations \eqref{Enmex2} and \eqref{Dnex2}, we write that
	\begin{multline}
		E_{n,m} D_{n,m}= \sum_{i=1}^n\sum_{j=1}^m \sum_{k=1}^n \left(\vec{e}_{n+m-1,k}\otimes \vec{e}_{m,1}^{\top}\otimes \vec{e}_{n,k}^{\top}\right)\left( \vec{e}_{n+m-1,i+j-1}^{\top}\otimes\vec{e}_{m,j}\otimes \vec{e}_{n,i}\right)+\\
		+\sum_{i=1}^n\sum_{j=1}^m\sum_{k=2}^m \left(\vec{e}_{n+m-1,n+k-1}\otimes\vec{e}_{m,k}^{\top}\otimes \vec{e}_{n,n}^{\top}\right)\left(\vec{e}_{n+m-1,i+j-1}^{\top}\otimes\vec{e}_{m,j}\otimes \vec{e}_{n,i} \right).\label{EnDn}
	\end{multline} 
	We denote with $\I_i\in \R^{(n+m-1)\times(n+m-1)}$ the matrix with $1$ in position $(i,i)$ and $0$ elsewhere, and with $\I^{i,j}\in \R^{(n+m-1)\times(n+m-1)}$ the matrix with $1$ in position $(k,k)$ for $i\le k\le j$. We focus on the first sum in equation \eqref{EnDn}. We notice that:
	\begin{itemize}[leftmargin=*]
		\item $\vec{e}_{m,1}^{\top}\otimes \vec{e}_{n,k}^{\top}=\vec{e}_{nm,k}^{\top}$ and $A_k := \vec{e}_{n+m-1,k}\otimes \vec{e}_{m,1}^{\top}\otimes \vec{e}_{n,k}^{\top}\in \R^{(n+m-1)\times nm}$ is the matrix with $1$ in position $(k,k)$ and $0$ elsewhere;
		\item $\vec{e}_{m,j}\otimes \vec{e}_{n,i}=\vec{e}_{nm,n(j-1)+i}$ and $B_{i,j} := \vec{e}_{n+m-1,i+j-1}^{\top}\otimes\vec{e}_{m,j}\otimes \vec{e}_{n,i}\in \R^{nm\times (n+m-1)}$ is the matrix with $1$ in position $(n(j-1)+i,i+j-1)$ and $0$ elsewhere.
	\end{itemize}
	Then $A_k B_{i,j} =\I_i$. Similarly, looking at the second sum in equation \eqref{EnDn}, we notice that $\vec{e}_{m,k}^{\top}\otimes \vec{e}_{n,n}^{\top}=\vec{e}_{nm,kn}^{\top}$ and $\tilde{A}_k := \vec{e}_{n+m-1,n+k-1}\otimes\vec{e}_{m,k}^{\top}\otimes \vec{e}_{n,n}^{\top}\in \R^{(n+m-1)\times nm}$ is the matrix with $1$ in position $(n+k-1,kn)$ and $0$ elsewhere. Then $\tilde{A}_k B_{i,j} =\I_{n+k-1}$. Combining these results into equation \eqref{EnDn} we get
	\begin{align*}
		E_{n, m}D_{n,m} &= \sum_{i=1}^n\sum_{j=1}^m \sum_{k=1}^n  A_k B_{i,j} + \sum_{i=1}^n\sum_{j=1}^m\sum_{k=2}^m \tilde{A}_k B_{i,j} \\&= \sum_{k=1}^{n}\I_k + \sum_{k=2}^{m} \I_{n+k-1} = \I^{1,n} + \I^{n+1, n+m-1} = I_{n+m-1},
	\end{align*}
	that concludes the proof.
\end{proof}

\subsubsection*{Proof of Proposition \ref{Mnex}}
\begin{proof}
By combining $\G H_{2n}(x) = G_{2n} H_{2n}(x)$ with Lemma \ref{id} we get
$vecL(\G X_n(x)) = G_{2n} vecL(X_n(x)),$ 
which, by definition of $E_{n+1}$ and multiplying both sides by $D_{n+1}$ from the left, becomes 	\begin{equation*}
		\label{Dd}
		D_{n+1} E_{n+1} vec(\G X_n(x)) = D_{n+1} G_{2n} E_{n+1} vec(X_n(x)).
	\end{equation*}
	In particular, $D_{n+1} E_{n+1} vec(\G X_n(x))  = vec(\G X_n(x))$ by Proposition \ref{EDp}, so that $\tilde{G}_n^{(1)}= D_{n+1} G_{2n} E_{n+1}$. We consider now the definition of exponential function as infinite sum of powers and write
	\begin{equation*}
		e^{\tilde{G}_n^{(1)} t} = e^{D_{n+1} G_{2n} E_{n+1}t} = \sum_{k=0}^{\infty}\frac{t^k}{k!}\left(D_{n+1} G_{2n} E_{n+1}\right)^k.
	\end{equation*} 
	The result follows from Proposition \ref{EDp} since $\left(D_{n+1} G_{2n} E_{n+1}\right)^k = D_{n+1} \left(G_{2n}\right)^k E_{n+1}$, $k\ge0$.
\end{proof}

\subsubsection*{Proof of Theorem \ref{odesol2}}
\begin{proof}
	Starting from equation \eqref{ode} and applying the operator $vec$ on both sides, we get
	\begin{equation*}
		\E\left[vec\left(X_{n}(Y(s_0))\right)\left.\right|\F_t\right] = vec\left(X_{n}(Y(t))\right) + \int_t^{s_0} \E\left[\G vec\left(X_{n}(Y(s))\right)\left.\right|\F_t\right] ds,
	\end{equation*}
	which, by equation \eqref{Mn} and Proposition \ref{Mnex}, becomes
	\begin{equation*}
		\E\left[vec\left(X_{n}(Y(s_0))\right)\left.\right|\F_t\right] = vec\left(X_{n}(Y(t))\right) +D_{n+1} G_{2n} E_{n+1}\int_t^{s_0} \E\left[ vec\left(X_{n}(Y(s))\right)\left.\right|\F_t\right] ds.
	\end{equation*} 
	For $Z(s):=\E[vec\left(X_{n}(Y(s))\right)\left.\right|\F_t]$, the proof then proceeds in a similar way to the proof of Theorem \ref{genm} and the statement is proved by combining the result with equation \eqref{corr2}.
\end{proof}

\subsubsection*{Proof of Proposition \ref{shapeX}}
\begin{proof}
By definition of d-Kronecker product, since $H_n(x)\in \R^{n+1}$, one can verify that $X_n^{(r)}(x) \in \R^{(n+1)\times(n+1)^r}$. We then proceed by induction on $r\ge 1$.
\begin{itemize}[leftmargin=*]
\item $r=1$: we get that $X_n^{(1)}(x)= H_n(x)^{\top} \otimes H_n(x) = x^0X_n^{(1)}(x)\in \mathcal{A}_{n+1,n+1}$.
\item $r \to  r+1$: assuming the statement holds for $r$, from equation \eqref{Hr}, we see that we need to multiply the row vector $H_n(x)^{\top} = (1,x,\dots,x^n)$ in the Kronecker sense with the matrix $X_n^{(r)}(x)= H_n(x)^{\top} \otimes^{r}H_n(x)$, which we know satisfies the statement of the proposition.  That means that each of the $(n+1)^{r-1}$ blocks $B_{n,r}^{(k)}(x)=x^{j_k} X_n^{(1)}(x)$, for $j_k \in \{0, \dots, (r-1)n\}$, must be multiplied with each of the elements of the vector $H_n(x)$, namely with each power $x^{\alpha}$, $\alpha=0,\dots, n$. We can then say that there are $(n+1)^{r}$ blocks $B_{n,r+1}^{(k)}(x)\in \mathcal{A}_{n+1,n+1}$ and that for each block there exists an index $\gamma_k \in \{0, \dots, rn\}$ such that $B_{n,r+1}^{(k)}(x)=x^{\gamma_k} X_n^{(1)}(x)$.
This concludes the proof.
\end{itemize} 
\end{proof}

\begin{lemma}
\label{lemma}
It holds that $$vec(X_n^{(m)}(x)) = 
H_n(x)^{\otimes^{m+1}}.$$ Moreover, after removing all the duplicates from $vec(X_n^{(m)}(x))$, we are left with $H_{n(m+1)}(x)$.
\begin{proof}
The result follows from a direct verification.
\end{proof}
\end{lemma}

\begin{lemma}
\label{lemmb}
There exist an L-eliminating matrix $E_{nm+1,n+1}$ and an L-duplicating matrix $D_{nm+1,n+1}$ such that
\begin{subequations}
\begin{align}
&E_{nm+1,n+1}\left( H_n(x) \otimes H_{nm}(x)\right) = H_{n(m+1)}(x),\label{b1}\\
&D_{nm+1,n+1}H_{n(m+1)}(x) =  H_n(x) \otimes H_{nm}(x).\label{b2}
\end{align}
\end{subequations}
\begin{proof}
From a direct verification, it can be seen that
\begin{equation}
\label{eq2}
vec(H_n(x)^{\top} \otimes H_{nm}(x)) = H_n(x) \otimes H_{nm}(x) \quad \mbox{and} \quad vecL(H_n(x)^{\top} \otimes H_{nm}(x)) = H_{n(m+1)}(x).
\end{equation}
Then, from Theorem \ref{Enmth}, there exists an L-eliminating matrix $E_{nm+1,n+1}$ transforming the vectorization of $H_n(x)^{\top} \otimes H_{nm}(x)$ into its L-vectorization. By equation \eqref{eq2}, this is equivalent to saying that $E_{nm+1,n+1}$ maps $H_n(x) \otimes H_{nm}(x)$ to $H_{n(m+1)}(x)$, which is what claimed in equation \eqref{b1}. Similarly, by Theorem \ref{Dnth} there exists an L-duplicating matrix $D_{nm+1,n+1}$ satisfying equation \eqref{b2}.
\end{proof}
\end{lemma}

\subsubsection*{Proof of Proposition \ref{propEp}}
\begin{proof}
We proceed by induction on $m\ge 1$.
\begin{itemize}[leftmargin=*]
\item $m=1$: see Corollary \ref{cor11} and equation \eqref{Xn1}.
\item $m-1 \to m$: we assume the statement holds for $m-1$, namely there exists a matrix $E_{n+1}^{(m-1)}$ that applied to $vec(X_n^{(m-1)}(x))$ removes all the duplicates. By Lemma \ref{lemma}, this means that$$E_{n+1}^{(m-1)}H_n(x)^{\otimes m}=H_{nm}(x).$$ We now multiply both sides in the Kronecker sense by $H_n(x)$, and successively apply on the left the matrix $E_{nm+1, n+1}$, obtaining that
\begin{equation*}
\label{stepp}
E_{nm+1, n+1}\left[H_n(x)\otimes \left(E_{n+1}^{(m-1)}H_n(x)^{\otimes m}\right)\right] =E_{nm+1, n+1}\left[H_n(x)\otimes H_{nm}(x)\right].
\end{equation*}
From the identity $H_n(x) = I_{n+1}H_n(x)$, and applying the mixed-product property of the Kronecker product (equation  \eqref{mix}) on the left hand side, and equation \eqref{b1} on the right hand side, we get
\begin{equation*}
E_{nm+1, n+1}\left[\left(I_{n+1}\otimes E_{n+1}^{(m-1)}\right)\left(H_n(x)\otimes H_n(x)^{\otimes m}\right)\right] =H_{n(m+1)}(x).
\end{equation*}
Since $H_n(x)\otimes H_n(x)^{\otimes m} =  H_n(x)^{\otimes (m+1)}=vec(X_n^{(m)}(x))$ by Lemma \ref{lemma}, the matrix $E_{n+1}^{(m)} = E_{nm+1, n+1}\left(I_{n+1} \otimes E_{n+1}^{(m-1)}\right)$ is exactly the one removing all the duplicates from $vec(X_n^{(m)}(x))$.
\end{itemize}
\end{proof}

\subsubsection*{Proof of Proposition \ref{propDp}}
\begin{proof}
We proceed by induction on $m\ge 1$.
\begin{itemize}[leftmargin=*]
\item $m=1$: see Corollary \ref{cor22} and equation \eqref{Xn1}.
\item $m-1 \to m$: we assume the statement holds for $m-1$. Then, starting from equation \eqref{b2} and multiplying both sides with $I_{n+1} \otimes D_{n+1}^{(m-1)}$ we get
\begin{equation}
\label{stepD}
\left(I_{n+1} \otimes D_{n+1}^{(m-1)}\right)\left(D_{nm+1,n+1}H_{n(m+1)}(x)\right) =  \left(I_{n+1} \otimes D_{n+1}^{(m-1)}\right) \left(H_n(x) \otimes H_{nm}(x)\right).
\end{equation}
By the mixed-product property of the Kronecker product (equation \eqref{mix}), the right hand side is 
\begin{align*}
\left(I_{n+1} \otimes D_{n+1}^{(m-1)}\right) \left(H_n(x) \otimes H_{nm}(x)\right)&=\left(I_{n+1} H_n(x)\right)\otimes \left(D_{n+1}^{(m-1)} H_{nm}(x)\right)\\&=H_n(x)\otimes \left(D_{n+1}^{(m-1)} H_{nm}(x)\right).
\end{align*} 
From the induction hypothesis, $D_{n+1}^{(m-1)}$ satisfies $D_{n+1}^{(m-1)} H_{nm}(x) = vec(X_n^{(m-1)}(x))$, and, by Lemma \ref{lemma}, we also have $vec(X_n^{(m-1)}(x))=H_{n}(x)^{\otimes m}$. Then equation \eqref{stepD} becomes
\begin{equation*}
\left(I_{n+1} \otimes D_{n+1}^{(m-1)}\right)D_{nm+1,n+1}H_{n(m+1)}(x) =  H_n(x)\otimes  H_{n}(x)^{\otimes m}=H_{n}(x)^{\otimes (m+1)},
\end{equation*}
and the matrix $D_{n+1}^{(m)} = \left(I_{n+1} \otimes D_{n+1}^{(m-1)}\right)D_{nm+1, n+1}$ is exactly the one required.
\end{itemize}
\end{proof}

\begin{proposition}
\label{mixpro}
For every $n,m\ge 1$, $\vec{v}_{n}\in\R^{n+1}$ and $M_n^{(m-1)}\in\R^{(n+1)^m\times(n+1)^m}$, the identity holds:
\begin{equation}
\label{ide}
H_n(x)\vec{v}_{n}^{\top}\left\{ vec^{-1} \circ M_n^{(m-1)}\circ vec\, \left(X_n^{(m-1)}(x)\right)\right\} = 
\left\{ X_n^{(m)}(x)\right\}M_n^{(m-1)\top}\left\{I_{n+1}\otimes^{m-1}\vec{v}_{n}\right\}.
\end{equation}

\begin{proof}
We proceed by induction on the order $m\ge 1$.
\begin{itemize}[leftmargin=*]
\item $m=1$: starting from the left hand side of identity \eqref{ide}, we get
\begin{align*}
	H_n(x)\vec{v}_{n}^{\top}\left\{ vec^{-1} \circ M_n^{(0)}\circ vec\, \left(X_n^{(0)}(x)\right)\right\}&
	= H_n(x)\vec{v}_{n}^{\top} M_n^{(0)} H_n(x)=H_n(x)H_n(x)^{\top} M_n^{(0)\top} \vec{v}_{n}\\&= \left\{ X_n^{(1)}(x)\right\}M_n^{(0)\top}\left\{I_{n+1}\otimes^{0}\vec{v}_{n}\right\}.
\end{align*}
Remember indeed that the $vec^{-1}$ operator transforms a vector into an object with the same dimension as the argument of the operator $vec$ previously applied. But in this case the argument of $vec$ is a vector already, hence both $vec$ and $vec^{-1}$ coincide in practise with the identity operator. Moreover, $\vec{v}_{n}^{\top} M_n^{(0)} H_n(x)\in \R$ and it equals its transpose. This proves the base case.

\item $m\to m+1$: we assume identity \eqref{ide} holds for $m$ and consider $$H_n(x)\vec{v}_{n}^{\top}\left\{ vec^{-1} \circ M_n^{(m)}\circ vec\, \left(X_n^{(m)}(x)\right)\right\}.$$ In particular, $M_n^{(m)}\in \R^{(n+1)^{m+1}\times(n+1)^{m+1}}$ can be seen as made up of $(n+1)^2$ matrices of the form $M_{i,j}^{(m-1)}\in \R^{(n+1)^m\times(n+1)^m}$, $1\le i,j \le n+1$, so that $M_n^{(m)}$ looks like
\begin{equation*}
	\footnotesize
	\NiceMatrixOptions{transparent}
	M_n^{(m)} = \begin{pmatrix}
		M_{1,1}^{(m-1)} & \cdots & M_{1,n+1}^{(m-1)}\\
		\vdots  &  \ddots & \vdots\\
		M_{n+1,1}^{(m-1)} &   \cdots & M_{n+1,n+1}^{(m-1)}
	\end{pmatrix}.
\end{equation*}
The idea is then to break up the matrix $M_n^{(m)}$ into sub-matrices, for which we know the statement holds by induction hypothesis. In what follows, starting from $X_n^{(m)}(x)$, we will apply in the following order: the $vec$ operator, the matrix $M_n^{(m)}$, the $vec^{-1}$ operator and finally the matrix $H_n(x)\vec{v}_{n}^{\top}$. At this point we will be able to apply the induction hypothesis, and prove the statement.

By Lemma \ref{lemma} and associativity property of the Kronecker product, we get
\begin{equation*}
	\footnotesize
	\NiceMatrixOptions{transparent}
	vec\left(X_n^{(m)}(x)\right) = 
	H_n(x)\otimes  H_n(x)^{\otimes m} =  \begin{pmatrix}
		H_n(x)^{\otimes m}\\
		xH_n(x)^{\otimes m}\\
		\vdots\\
		x^n H_n(x)^{\otimes m}
	\end{pmatrix},
\end{equation*}
where $x^k H_n(x)^{\otimes m}\in \R^{(n+1)^m}$, $k=0, \dots, n$, thus 
\begin{equation*}
	\footnotesize
	M_n^{(m)} vec \left(X_n^{(m)}(x)\right)=
	\begin{pNiceMatrix}
		\footnotesize
		\NiceMatrixOptions{transparent}
		M_{1,1}^{(m-1)}H_n(x)^{\otimes m} +  M_{1,2}^{(m-1)}xH_n(x)^{\otimes m}+ \cdots + M_{1,n+1}^{(m-1)}x^nH_n(x)^{\otimes m}\\
		\Vdots \\
		M_{n+1,1}^{(m-1)}H_n(x)^{\otimes m} +  M_{n+1,2}^{(m-1)}xH_n(x)^{\otimes m}+ \cdots+ M_{n+1,n+1}^{(m-1)}x^nH_n(x)^{\otimes m}
	\end{pNiceMatrix}.
\end{equation*}
Applying the $vec^{-1}$ operator to the last matrix obtained, by linearity we get
\begin{multline*}
	\footnotesize
	\left(vec^{-1}\left(M_{1,1}^{(m-1)}H_n(x)^{\otimes m}\right) + 
	\cdots + x^nvec^{-1}\left(M_{1,n+1}^{(m-1)}H_n(x)^{\otimes m}\right), \right.\\\cdots\cdots,\\
	\left. vec^{-1}\left(M_{n+1,1}^{(m-1)}H_n(x)^{\otimes m}\right) + 
	\cdots + x^nvec^{-1}\left(M_{n+1,n+1}^{(m-1)}H_n(x)^{\otimes m}\right) \right),
\end{multline*}
where $vec^{-1}\left(M_{i,j}^{(m-1)}H_n(x)^{\otimes m}\right)\in \R^{(n+1)\times (n+1)^{m-1}}$, $1\le i,j \le n+1$. Multiplying the above equation by $H_n(x)\vec{v}_{n}^{\top}$, we obtain that
\begin{align*}
	&H_n(x)\vec{v}_{n}^{\top}\left\{ vec^{-1} \circ M_n^{(m)}\circ vec\, \left(X_n^{(m)}(x)\right)\right\} \\
	&= \left(H_n(x)\vec{v}_{n}^{\top}vec^{-1}\left(M_{1,1}^{(m-1)}H_n(x)^{\otimes m}\right) + 
	\cdots + x^nH_n(x)\vec{v}_{n}^{\top}vec^{-1}\left(M_{1,n+1}^{(m-1)}H_n(x)^{\otimes m}\right), \right.\\
	&\qquad \qquad \qquad \qquad \qquad \qquad \cdots\cdots,\\
	&\qquad \left. H_n(x)\vec{v}_{n}^{\top}vec^{-1}\left(M_{n+1,1}^{(m-1)}H_n(x)^{\otimes m}\right) + 
	\cdots + x^nH_n(x)\vec{v}_{n}^{\top}vec^{-1}\left(M_{n+1,n+1}^{(m-1)}H_n(x)^{\otimes m}\right) \right).
\end{align*}
Since $H_n(x)^{\otimes m} = vec\left(X_n^{(m-1)}(x) \right)$, we apply the induction hypothesis to each term
\begin{equation*}
	H_n(x)\vec{v}_{n}^{\top}vec^{-1}\left(M_{i,j}^{(m-1)}H_n(x)^{\otimes m}\right) = \left\{ X_n^{(m)}(x)\right\}M_{i,j}^{(m-1)\top}\left\{I_{n+1}\otimes^{m-1}\vec{v}_{n}\right\}
\end{equation*}
for $1\le i,j \le n+1$, so that
\begin{align*}
	&H_n(x)\vec{v}_{n}^{\top}\left\{ vec^{-1} \circ M_n^{(m)}\circ vec\, \left(X_n^{(m)}(x)\right)\right\}= \\
	&\left(\left\{ X_n^{(m)}(x)\right\}M_{1,1}^{(m-1)\top}\left\{I_{n+1}\otimes^{m-1}\vec{v}_{n}\right\}  + 
	\cdots +x^n\left\{ X_n^{(m)}(x)\right\}M_{1,n+1}^{(m-1)\top}\left\{I_{n+1}\otimes^{m-1}\vec{v}_{n}\right\}, \right.\\
	&\quad\qquad\qquad\qquad\qquad\cdots\cdots,\\
	&\left. \left\{ X_n^{(m)}(x)\right\}M_{n+1,1}^{(m-1)\top}\left\{I_{n+1}\otimes^{m-1}\vec{v}_{n}\right\}+ 
	\cdots +x^n\left\{ X_n^{(m)}(x)\right\}M_{n+1, n+1}^{(m-1)\top}\left\{I_{n+1}\otimes^{m-1}\vec{v}_{n}\right\} \right)
\end{align*}
which can also be seen as the following matrix product:
\begin{equation*}
	\footnotesize
	\begin{pNiceMatrix}
		\left\{ X_n^{(m)}(x)\right\}, & \Cdots & , x^n\left\{ X_n^{(m)}(x)\right\}
	\end{pNiceMatrix}
	\begin{pNiceMatrix}
		M_{1,1}^{(m-1)\top} \left\{I_{n+1}\otimes^{m-1}\vec{v}_{n}\right\}& \Cdots & M_{n+1,1}^{(m-1)\top}\left\{I_{n+1}\otimes^{m-1}\vec{v}_{n}\right\}\\
		\Vdots  &  \Ddots & \Vdots\\
		M_{1,n+1}^{(m-1)\top}\left\{I_{n+1}\otimes^{m-1}\vec{v}_{n}\right\} &   \Cdots & M_{n+1,n+1}^{(m-1)\top}\left\{I_{n+1}\otimes^{m-1}\vec{v}_{n}\right\}
	\end{pNiceMatrix}.
\end{equation*}
In particular, the following identity also holds
\begin{equation*}
	\left(\left\{ X_n^{(m)}(x)\right\}, \cdots, x^n\left\{ X_n^{(m)}(x)\right\}\right) = H_n(x)^{\top}\otimes \left\{ X_n^{(m)}(x)\right\} =  X_n^{(m+1)}(x),
\end{equation*}
so that we can conclude with the expression
\begin{align*}
	\NiceMatrixOptions{transparent}
	&\begin{pNiceMatrix}
		\NiceMatrixOptions{transparent}
		M_{1,1}^{(m-1)\top} \left\{I_{n+1}\otimes^{m-1}\vec{v}_{n}\right\}& \Cdots & M_{n+1,1}^{(m-1)\top}\left\{I_{n+1}\otimes^{m-1}\vec{v}_{n}\right\}\\
		\Vdots  &  \Ddots & \Vdots\\
		M_{1,n+1}^{(m-1)\top}\left\{I_{n+1}\otimes^{m-1}\vec{v}_{n}\right\} &   \Cdots & M_{n+1,n+1}^{(m-1)\top}\left\{I_{n+1}\otimes^{m-1}\vec{v}_{n}\right\}
	\end{pNiceMatrix}=\\
	& \begin{pNiceMatrix}
		\NiceMatrixOptions{transparent}
		M_{1,1}^{(m-1)\top} & \Cdots & M_{n+1,1}^{(m-1)\top}\\
		\Vdots  &  \Ddots & \Vdots\\
		M_{1,n+1}^{(m-1)\top}&   \Cdots & M_{n+1,n+1}^{(m-1)\top}
	\end{pNiceMatrix}\!
	\begin{pNiceMatrix}
		\left\{I_{n+1}\otimes^{m-1}\vec{v}_{n}\right\}& \Cdots & 0\\
		\Vdots &\Ddots& \Vdots\\
		0  & \Cdots &\left\{I_{n+1}\otimes^{m-1}\vec{v}_{n}\right\}
	\end{pNiceMatrix}
\end{align*}
where the first matrix on the right hand side coincides with $M_n^{(m)\top}$, while the second is $I_{n+1}\otimes \left\{I_{n+1}\otimes^{m-1}\vec{v}_{n}\right\} = I_{n+1}\otimes^{m}\vec{v}_{n}$. This means we proved
\begin{equation*}
	H_n(x)\vec{v}_{n}^{\top}\left\{ vec^{-1} \circ M_n^{(m)}\circ vec\, \left(X_n^{(m)}(x)\right)\right\} = \left\{X_n^{(m+1)}(x)\right\}M_n^{(m)\top}\left\{I_{n+1}\otimes^{m}\vec{v}_{n}\right\},
\end{equation*}
and therefore reached the claim.
\end{itemize}
\end{proof}
\end{proposition}

\subsubsection*{Proof of Theorem \ref{theoremformula}}
\begin{proof}
Following the same idea as in the proof of Theorem \ref{odesol2}, we start by proving that
\begin{equation}
	\label{expr}
	\E\left[\left.X_n^{(r)}(Y(s))\right|\F_t\right]\\=vec^{-1}\circ e^{\tilde{G}_n^{(r)}(s-t)}\circ vec \,\left( X_n^{(r)}(Y(t))\right)
\end{equation}
for $\tilde{G}_n^{(r)}=D_{n+1}^{(r)}G_{n(r+1)}E_{n+1}^{(r)}$.  By equation \eqref{martingale} we write
\begin{equation*}
\E\left[\left.X_n^{(r)}(Y(s))\right|\F_t\right] = X_n^{(r)}(Y(t))+ \int_t^{s} \E\left[\left.\G \left(X_n^{(r)}(Y(u))\right)\right|\F_t\right] du,
\end{equation*} 
and applying the $vec$ operator on both sides we get
\begin{equation}
\label{oder}
\E\left[\left.vec\left(X_n^{(r)}(Y(s))\right)\right|\F_t\right] = vec\left(X_n^{(r)}(Y(t))\right)+ \int_t^{s} \E\left[\left.\G vec\left(X_n^{(r)}(Y(u))\right)\right|\F_t\right] du.
\end{equation} 
By Proposition \ref{propEp}, there exists an $r$-th L-eliminating matrix $E_{n+1}^{(r)}$ such that
\begin{equation}
\label{1}
E_{n+1}^{(r)}vec\left(X_n^{(r)}(x)\right) = H_{n(r+1)}(x).
\end{equation}
From Theorem \ref{genm}, there also exists a generator matrix $G_{n(r+1)}$ such that 
\begin{equation}
\label{2}
\G H_{n(r+1)}(x) = G_{n(r+1)}H_{n(r+1)}(x),
\end{equation}
and by Proposition \ref{propDp}, an $r$-th L-duplicating matrix $D_{n+1}^{(r)}$ such that
\begin{equation}
\label{3}
D_{n+1}^{(r)}H_{n(r+1)}(x) = vec\left(X_n^{(r)}(x)\right).
\end{equation}
Combining equations \eqref{1}, \eqref{2}, \eqref{3} with equation \eqref{oder} we get
\begin{equation*}
\E\left[\left.vec\left(X_n^{(r)}(Y(s))\right)\right|\F_t\right] = vec\left(X_n^{(r)}(Y(t))\right)+ D_{n+1}^{(r)}G_{n(r+1)}E_{n+1}^{(r)}\int_t^{s} \E\left[\left.vec\left(X_n^{(r)}(Y(u))\right)\right|\F_t\right] du,
\end{equation*} 
and, proceeding the proof as in Theorem \ref{odesol2}, we obtain that for every $n\ge 1$ and $r\ge 0$, the  matrix $\tilde{G}_n^{(r)}=D_{n+1}^{(r)}G_{n(r+1)}E_{n+1}^{(r)}\in \R^{(n+1)^{r+1}\times(n+1)^{r+1}}$ is such that the expectation formula  \eqref{expr} holds.

We now proceed by induction on the number of polynomials $m\ge 1$ to prove the correlator formula.
\begin{itemize}[leftmargin=*]
\item $m=1$: the formula coincides with the one given in Theorem \ref{odesol2}.
\item $m\to m+1$: we suppose the correlator formula holds for $m$ and consider $m+1$ polynomial functions. By the tower rule and the induction hypothesis, we write that
\begin{align*}
	&C_{p_0,\dots, p_{m+1}}(s_0,\dots,s_{m+1};t) =
	\E\left[\left.p_{m+1}\left(Y(s_0)\right) C_{p_0,\dots, p_m}(s_1,\dots, s_{m+1};s_0)   \right|\F_t\right]\\
&=\E\left[\vec{p}_{m+1}^{\top}H_n(Y(s_0))\vec{p}_{m}^{\top} \left\{ vec^{-1} \circ e^{\tilde{G}_n^{(m)}(s_1-s_0)}\circ vec\left( X_n^{(m)}(Y(s_0))\right)\right\} \cdot \right.\\ &\qquad  \quad \qquad  \qquad \qquad \qquad \qquad \qquad \qquad  \left.\cdot \left.\prod_{k=1}^{m}e^{\tilde{G}_n^{(m-k)\top}(s_{k+1}-s_{k})}\left\{I_{n+1}\otimes^{m-k}\vec{p}_{m-k} \right\}\right|\F_t\right].
\end{align*}
From Proposition \ref{mixpro}, we also have the following crucial equality:
\begin{multline*}
	H_n(Y(s_0))\vec{p}_{m}^{\top}\left\{ vec^{-1} \circ e^{\tilde{G}_n^{(m)}(s_1-s_0)}\circ vec\, \left(X_n^{(m)}(Y(s_0))\right)\right\} \\ = 
	\left\{ X_n^{(m+1)}(Y(s_0))\right\}e^{\tilde{G}_n^{(m)\top}(s_1-s_0)}\left\{I_{n+1}\otimes^{m}\vec{p}_{m}\right\}.
\end{multline*}
Combining the previous results and equation \eqref{expr}, we can write that
\begingroup
\allowdisplaybreaks
\begin{equation*}
	\begin{aligned}
		&C_{p_0,\dots, p_{m+1}}(s_0,\dots,s_{m+1};t) \\
		&=\E\left[\vec{p}_{m+1}^{\top}\left\{ X_n^{(m+1)}(Y(s_0))\right\}e^{\tilde{G}_n^{(m)\top}(s_1-s_0)}\left\{I_{n+1}\otimes^{m}\vec{p}_{m}\right\} \cdot\right.\\
		&\left.\qquad \qquad \qquad \qquad \qquad \qquad \qquad \qquad \qquad \qquad  \left.\cdot\prod_{k=1}^{m}e^{\tilde{G}_n^{(m-k)\top}(s_{k+1}-s_{k})}\left\{I_{n+1}\otimes^{m-k}\vec{p}_{m-k} \right\}\right|\F_t\right]\\
		&=\vec{p}_{m+1}^{\top}\E\left[ \left.X_n^{(m+1)}(Y(s_0))\right|\F_t\right] \prod_{k=0}^{m}e^{\tilde{G}_n^{(m-k)\top}(s_{k+1}-s_{k})}\left\{I_{n+1}\otimes^{m-k}\vec{p}_{m-k} \right\}\\
		&=\vec{p}_{m+1}^{\top}\left\{vec^{-1}\circ e^{\tilde{G}_n^{(m+1)}(s_0-t)}\circ vec \,\left( X_n^{(m+1)}(Y(t))\right)\right\} \prod_{k=0}^{m}e^{\tilde{G}_n^{(m-k)\top}(s_{k+1}-s_{k})}\left\{I_{n+1}\otimes^{m-k}\vec{p}_{m-k} \right\}.
	\end{aligned}
\end{equation*}
\endgroup
Rearranging the index in the product of the last equation, we get the formula for $m+1$ polynomial functions and conclude the proof.
\end{itemize}
\end{proof}

\subsubsection*{Proof of Theorem \ref{Gn}}
\begin{proof}
We proceed by induction on the dimension $n\ge 2$.
\begin{itemize}[leftmargin=*]
\item $n=2$: see Example \ref{G2ex}.
\item $n-1\to n$: we assume the recursion formula holds for $n-1$. We then need $\G x^n$. By equations \eqref{genD} and \eqref{poldefcond}, we write that
\begin{equation*}
\G x^n = n(b_0+b_1x)x^{n-1} +\frac{1}{2}n(n-1)(\sigma_0+\sigma_1x+\sigma_2x^2)x^{n-2}+\int_{\R}\left(\left(x+z
\right)^n-x^n-nx^{n-1}z\right)\ell(x,dz).
\end{equation*}
In particular, by binomial expansion
\begin{align*}
&\int_{\R}\left(\left(x+z
\right)^n-x^n-nx^{n-1}z\right)\ell(x,dz) =\sum_{k=2}^n\binom{n}{k}x^{n-k}\int_{\R}z^k\ell(x,dz)\\
&=\sum_{k=2}^n\sum_{i=0}^k\binom{n}{k}\xi_i^kx^{n-k+i}=\sum_{i=0}^n\left(\sum_{k=\max(2,i)}^n\binom{n}{k}\xi_{k-i}^k\right)x^{n-i},
\end{align*}
so that
\begin{equation*}
\G x^n = \left(nb_1+\frac{1}{2}n(n-1)\sigma_2\right)x^n + \left(nb_0 +\frac{1}{2}n(n-1)\sigma_1\right)x^{n-1} +\sum_{i=0}^n\left(\sum_{k=\max(2,i)}^n\binom{n}{k}\xi_{k-i}^k\right)x^{n-i},
\end{equation*}
which must be rearranged to collect the coefficients of $x^k$, $k=0, \dots, n$, to be inserted in the last row of $G_n$. This leads to $(a_n^n, a_n^{n-1},\dots, a_n^1, c_n)^{\top}$ as defined in the theorem and concludes the proof.
\end{itemize}
\end{proof}

\subsubsection*{Proof of Theorem \ref{Gnk}}
\begin{proof}
We first prove the base case $n=1$. We consider the definition of matrix exponential as infinite sum of powers. If $b_1\ne 0$ then for every $k\ge 1$,
$G_1^k =\begin{pmatrix}
	0 & 0 \\
	b_0b_1^{k-1} & b_1 ^k\end{pmatrix}$; if $b_1=0$ then for every $k\ge 2$, $G_1^k = \begin{pmatrix}
	0 & 0 \\
	0 & 0
\end{pmatrix}$. With $G_1^0 = I_{2}$, for $b_1\ne0$ we get:
\begin{align*}
	e^{G_1t} &=\sum_{k=0}^{\infty}\frac{\left(G_1t\right)^k}{k!}= I_{2}+ \sum_{k=1}^{\infty}\frac{t^k}{k!}\begin{pmatrix}
		0 & 0 \\
		b_0b_1^{k-1} & b_1 ^k 
	\end{pmatrix}\\&= 
	\begin{pmatrix}
		1 & 0 \\
		b_0\sum_{k=1}^{\infty}\frac{t^kb_1^{k-1}}{k!} & 1+\sum_{k=1}^{\infty}\frac{t^kb_1 ^k}{k!} 
	\end{pmatrix}=\begin{pmatrix}
		1 & 0 \\
		\frac{b_0}{b_1}\left(e^{b_1t}-1\right) & e^{b_1t}
	\end{pmatrix},
\end{align*}
and similarly for $b_1=0$, so that the base case is proved.
	
We now set $n> 1$. For $\Lambda_n := c_nI_{n}-G_{n-1}$, it is easy to verify that for every $k\ge 1$ the powers of $G_n$ are given by
\begin{equation*}
\label{Gnkeq}
G_n^k =\begin{pmatrix}
	G_{n-1}^k & \vec{0}_n \\
	\vec{a}_n^{\top}\Lambda_n^{-1}\left(c_n^kI_{n}-G_{n-1}^k\right) & c_n^k \\
\end{pmatrix},
\end{equation*}
provided that $\Lambda_n$ is invertible. More precisely, as a consequence of the recursion formula for $G_n$, the power $G_n^k$ involves $\Lambda_n^{-1}$, but it also involves $G_{n-1}^k$, that means it involves $\Lambda_{n-1}^{-1}$, and so on. Thus the matrices $\Lambda_r^{-1}$ are all involved in $G_n^k$ for every $2\le r \le n$. In particular, we need all these to be invertible. For a fixed $r$, the determinant of $\Lambda_r$ must then be different from zero for every $2\le r \le n$. In particular, since $G_{r-1}$ is a (lower) triangular matrix, $\Lambda_r$ is a (lower) triangular matrix with determinant given by the product of the elements on the main diagonal. By Theorem \ref{Gn} and equation \eqref{diag} we then get:
\begin{equation}
	\label{cond1}
	\det\left(c_rI_{r}-G_{r-1}\right) = c_r\prod_{j=1}^{r-1}\left(c_r-c_j\right)\ne 0
\end{equation} 
where $c_1 = b_1$. Condition \eqref{cond1} is equivalent to ask that $c_r \ne 0$ and $c_r\ne c_j$ for every $1\le j \le r-1$. One easily notice that asking these conditions to hold for every $2 \le r \le n$ means to ask that all the coefficients $\{c_j\}_{j=2}^n$ are not null, and that $\{c_j\}_{j=1}^n$ are all different among each others. This coincides with condition \eqref{lambdacond}. 

Then, with $G_n^0 = I_{n+1}$, we get:
\begin{align*}
e^{G_nt} &=\sum_{k=0}^{\infty}\frac{\left(G_nt\right)^k}{k!}=I_{n+1} + \sum_{k=1}^{\infty}\frac{t^k}{k!}\begin{pmatrix}
	G_{n-1}^k & \vec{0}_n \\
	\vec{a}_n^{\top}\Lambda_n^{-1}\left(c_n^kI_{n}-G_{n-1}^k\right)& c_n^k 
\end{pmatrix}\\
&=\begin{pmatrix}
	I_{n} + \sum_{k=1}^{\infty}\frac{\left(G_{n-1}t\right)^k}{k!} & 0 \\
	\vec{a}_n^{\top}\Lambda_n^{-1}\sum_{k=1}^{\infty}\frac{t^k}{k!} \left(c_n^kI_{n}-G_{n-1}^k\right)& 1+\sum_{k=1}^{\infty}\frac{t^kc_n ^k}{k!} 
\end{pmatrix} \\&= \begin{pmatrix}
	e^{G_{n-1}t} & \vec{0}_n \\
	\vec{a}_n^{\top}\Lambda_n^{-1}\left(I_{n}\sum_{k=1}^{\infty}\frac{(c_nt)^k}{k!}-\sum_{k=1}^{\infty}\frac{(G_{n-1}t)^k}{k!}\right) & e^{c_nt} 
\end{pmatrix}=\begin{pmatrix}
	e^{G_{n-1}t} & \vec{0}_n \\
	\vec{a}_n^{\top}\Lambda_n^{-1}\left(e^{c_nt}I_{n}-e^{G_{n-1}t}\right) & e^{c_nt} 
\end{pmatrix}
\end{align*}
which proves the matrix exponential formula and concludes the proof.
\end{proof}

\subsubsection*{Proof of Lemma \ref{corrr}}
\begin{proof}
	We consider the definition of $c_j$ in equation \eqref{an}: if $\ell(x,dz)\equiv0$, then for every $j=2, \dots, n$, $c_j= jb_1 +\frac{1}{2}j(j-1)\sigma_2,$ so that the first condition in \eqref{lambdacond}
	equals $b_1 \ne -\frac{(j-1)}{2}\sigma_2$, equivalent also to 
	\begin{equation}
		\label{condl1}
		b_1 \ne -\frac{k}{2}\sigma_2, \; \mbox{for every }\; 1 \le k \le n-1.
	\end{equation}
	In the second condition in \eqref{lambdacond}, we require $c_j \ne c_i$, $1 \le j  < i\le n$, that is
	\begin{equation*}
		jb_1 +\frac{1}{2}j(j-1)\sigma_2 \ne ib_1 +\frac{1}{2}i(i-1)\sigma_2, \quad \mbox{for every }\; 1 \le j  < i\le n,
	\end{equation*}
	which, after some simplifications, can be rewritten as $b_1 \ne -\frac{(j+i-1)}{2}\sigma_2$ for every $1 \le j  < i\le n$. In particular, since $3 \le j+i \le 2n-1$, this is also equivalent to
	\begin{equation*}
		b_1 \ne -\frac{k}{2}\sigma_2, \quad \mbox{for every }\; 2 \le k\le 2(n-1).
	\end{equation*}
	Adding this to the condition previously found in \eqref{condl1}, we conclude the proof.
\end{proof}

\subsubsection*{Proof of Proposition \ref{key}}
\begin{proof}
By equations \eqref{MM} and \eqref{Q} and linearity of the extended generator $\G$, we write the equalities
\begin{equation*}
	\label{equ}
	J_n M_nH_n(x) = J_n Q_n(x)  = \G Q_n(x) = \G\left(M_nH_n(x)\right)= M_n\G H_n(x) = M_nG_n H_n(x).
\end{equation*}
By comparing the first and the last terms we get $J_n M_n = M_nG_n$, which, rearranged, gives the first equality of the proposition. As a direct consequence, from the definition of exponential function as infinite sum of powers and by the identity $M_n^{-1}M_n=I_n$, we get also the second equality.
\end{proof}

\subsubsection*{Proof of Proposition \ref{propdelta}}
\begin{proof}
From Theorem \ref{theoremformula}, we easily notice that the only dependence of $C_{p_0,\dots, p_m}(s_0,\dots, s_m;t)$ on the initial condition $Y(t)$ is inside the matrix of function $X_n^{(m)}(Y(t))$, which is defined by equation \eqref{Xnm} as the $m$-th Kronecker product $X_n^{(m)}(x)=H_n(x)^{\top} \otimes^mH_n(x)$. In particular, by associativity property of the Kronecker product, we can also write that $X_n^{(m)}(x)=H_n(x)^{\top} \otimes X_n^{(m-1)}(x)$. By the product rule  applied to $H_n(Y(t))^{\top} \otimes X_n^{(m-1)}(Y(t))$, we obtain then the recursive formula for  $\frac{\partial X_n^{(m)}(Y(t))}{\partial Y(t)}$ as in the statement, where $$\frac{\partial H_n(Y(t))}{\partial Y(t)} = \frac{\partial X_n^{(0)}(Y(t))}{\partial Y(t)}  =\left(0, 1, 2Y(t), \dots, nY(t)^{n-1}\right).$$ In particular, this can be seen as the vector product of $\left(0,1,\dots, n\right)$ with $ \left(0, H_{n-1}(Y(t))^{\top}\right)^{\top}$. This concludes the proof.
\end{proof}

\subsubsection*{Proof of Proposition \ref{proptheta}}
\begin{proof}
From the correlator formula in Theorem \ref{theoremformula}, we distinguish three different cases, namely $j=0$, $1\le j < m$, and $j=m$, which we shall analyse separately.

For $j=0$, the time point $s_0$ appears in $C_{p_0,\dots, p_m}(s_0,\dots, s_m;t)$ two times: once inside the curl parenthesis in the matrix exponential $e^{\tilde{G}_n^{(m)}(s_0-t)}$, and once in the product $\prod_{k=1}^m$ for $k=1$ in the matrix exponential $e^{\tilde{G}_n^{(m-1)\top}(s_1-s_0)}$. By the product rule, one gets $\Theta_0$.

For $1\le j < m$, the time point $s_j$ also appears in $C_{p_0,\dots, p_m}(s_0,\dots, s_m;t)$ two times, both in the product $\prod_{k=1}^m$, first for $k=j$ in the matrix exponential $e^{\tilde{G}_n^{(m-j)\top}(s_j-s_{j-1})}$, and then for $k=j+1$ in the matrix exponential $e^{\tilde{G}_n^{(m-j-1)\top}(s_{j+1}-s_{j})}$. By the product rule, one gets $\Theta_j$.

For $j=m$, the time point $s_m$ appears in $C_{p_0,\dots, p_m}(s_0,\dots, s_m;t)$ only one time, that is in the product $\prod_{k=1}^m$ for $k=m$ in the matrix exponential $e^{\tilde{G}_n^{(0)\top}(s_m-s_{m-1})}$. By differentiation, one gets $\Theta_m$.

This concludes the proof.
\end{proof}

\end{document}